\documentclass[11pt]{article}

\usepackage{amsmath}
\usepackage{amssymb}
\usepackage{amsthm}
\usepackage{xfrac}
\usepackage{amsfonts}
\usepackage{dsfont}
\usepackage[english]{babel}
\usepackage{bbm}
\usepackage{bm}
\usepackage{booktabs} 
\usepackage{caption}
\usepackage{cite}
\usepackage{enumerate}
\usepackage{epigraph}
\usepackage{etoolbox} 
\usepackage{float}
\usepackage{graphicx}
\usepackage{mathtools}
\usepackage{pgfplots}
\usepackage{float}
\usepackage{mathtools}
\usepackage{microtype} 
\usepackage{multicol}
\usepackage{gauss} 
\usepackage{geometry}
\usepackage{todonotes}
\usepackage[perpage,para,symbol*]{footmisc}
\usepackage{subfigure}
\usepackage{tabularx}
\usepackage{tikz}
\usepackage{url}
\usepackage{verbatim}
\usepackage{xparse}
\allowdisplaybreaks

\newcolumntype{x}[1]{!{\centering\arraybackslash\vrule width #1}}

\newtheorem{theorem}{Theorem}

\newtheorem{remark}{Remark}
\newtheorem{corollary}{Corollary}
\newtheorem{definition}{Definition}

\newtheorem{proposition}{Proposition}
\newtheorem{example}{Example}

\newtheorem{construction}{Construction}

\newcommand{\N}{{\mathbb N}}
\newcommand{\R}{{\mathbb R}}

\newcommand{\Z}{{\mathbb Z}}
\newcommand{\E}{{\mathbb E}}
\newcommand{\tr}{\text{tr}}
\newcommand{\C}{\mathbb C}
\newcommand*\diff{\mathop{}\!\mathrm{d}}

\newcommand{\eqd}{\stackrel{d}{=}}

\NewDocumentCommand{\ceil}{s O{} m}{%
  \IfBooleanTF{#1} % starred
    {\left\lceil#3\right\rceil} % \ceil*[..]{..}
    {#2\lceil#3#2\rceil} % \ceil[..]{..}
}

\definecolor{edgecolour}{RGB}{39,64,112}
\definecolor{dualvertexcolour_pentagon}{RGB}{205,79,57}
\definecolor{dualvertexcolour_hexagon}{RGB}{79,205,57}
\tikzstyle{pentagon}=[circle,draw,inner sep=0pt, fill=dualvertexcolour_pentagon, minimum width=1.5mm]
\tikzstyle{hexagon}=[circle,draw,inner sep=0pt, fill=dualvertexcolour_hexagon, minimum width=1.5mm]
\tikzstyle{facet}=[circle,draw,inner sep=0pt, minimum width=1.5mm]
\tikzstyle{edge}=[draw, color=edgecolour, line width=0.3mm]

\makeatletter
\let\@fnsymbol\@arabic
\newcommand{\specificthanks}[1]{\@fnsymbol{#1}}
\makeatother

\begin{document}

\title{\textbf{Random eigenvalues of nanotubes}}
\insert\footins{\footnotesize \noindent Artur Bille \hspace*{3cm} Pavel Ievlev \\ 
\texttt{artur.bille@uni-ulm.de} \hspace*{0.9cm} \texttt{ievlev.pn@gmail.com}  \vspace*{0.2cm} \\  
Victor Buchstaber \hspace*{2cm} Svyatoslav Novikov \hspace*{2cm} Evgeny Spodarev\\ 
\texttt{buchstab@mi-ras.ru} \hspace*{1.6cm} \texttt{ Svyatoslav.Novikov@unil.ch}\hspace*{1.5cm} \texttt{evgeny.spodarev@uni-ulm.de}  \vspace*{0.2cm}}

\author{Artur Bille\thanks{Ulm University, Germany \quad\quad *Corresponding author}\textsuperscript{\hspace{0.15cm}*}\and 
Victor Buchstaber\thanks{Steklov Mathematical Institute RAN, Russia} \and 
Pavel Ievlev\thanks{Université de Lausanne, Switzerland}\and 
Svyatoslav Novikov\textsuperscript{\specificthanks{3}} \and
Evgeny Spodarev\textsuperscript{\specificthanks{1}}
}
\date{}
\maketitle

\begin{abstract}
\label{abstract}
\noindent The hexagonal lattice and its dual, the triangular lattice, serve as powerful models for comprehending the atomic and ring connectivity, respectively, in \textit{graphene} and \textit{carbon $(p,q)$--nanotubes}.
The chemical and physical attributes of these two carbon allotropes are closely linked to the average number of closed paths of different lengths $k\in\N_0$ on their respective graph representations.
Considering that a carbon $(p,q)$--nanotube can be thought of as a graphene sheet rolled up in a matter determined by the \textit{chiral vector} $(p,q)$, our findings are based on the study of \textit{random eigenvalues} of both the hexagonal and triangular lattices presented in \cite{bille2023random}. 
This study reveals that for any given \textit{chiral vector} $(p,q)$, the sequence of counts of closed paths forms a moment sequence derived from a functional of two independent uniform distributions.
Explicit formulas for key characteristics of these distributions, including probability density function (PDF) and moment generating function (MGF), are presented for specific choices of the chiral vector.
Moreover, we demonstrate that as the \textit{circumference} of a $(p,q)$--nanotube approaches infinity, i.e., $p+q\rightarrow \infty$, the $(p,q)$--nanotube tends to converge to the hexagonal lattice with respect to the number of closed paths for any given length $k$, indicating weak convergence of the underlying distributions. 
\end{abstract}
\textbf{Keywords:}  graphene, regular hexagonal lattice, nanotube, method of moments, spectrum, random eigenvalue, characteristic function.
\noindent
\textbf{MSC2020:}  \textbf{Primary:} 05C10;  \textbf{Secondary:} 05C63,  05C38, 44A60, 92E10. 

% 05C10 -- Planar graphs; geometric and topological aspects of graph theory
% 05C63 -- Infinite graphs
% 05C38 -- Paths and cycles
% 44A60 -- Moment problems
% 33C05 -- Classical hypergeometric functions, ${}_2 F_1$
% 92E10 -- Molecular structure (graph-theoretic methods, methods of differential topology, etc.)

\section{Introduction}\label{sec:introduction}
Over the past three decades, significant progress has been made in the synthesis of novel carbon structures, resulting in the awarding of two Nobel Prizes. 
In 1996, Kroto, Curl, and Smalley were honored with the Nobel Prize in Chemistry for their pioneering roles in discovering \textit{fullerenes}, while in 2010, Novoselov and Geim received the Nobel Prize in Physics for their groundbreaking work in isolating and characterizing \textit{graphene}.
Another noteworthy carbon allotrope, \textit{carbon nanotube}, was also discovered in the previous century. 
Though there remains a debate over the exact pioneers of this discovery, it is widely accepted that Radushkevich and Lukyanovich's 1952 publication, featuring an explicit image of a carbon nanotube, marked a pivotal milestone.

%These achievements, along with many others in science, are built upon fundamental research conducted by numerous scientists before them. 
%For instance, Eiji Osawa's theoretical prediction of the \textit{Buckyball fullerene} $C_{60}$ in 1970, and Grimme et al.'s comprehensive analysis of all 1812 $C_{60}$--isomers in 2017, which established topological indices, geometrical measures, and theoretical electronic properties as \textit{good} stability predictors for fullerenes, formed the theoretical basis and motivation for future breakthroughs.

In this paper, we present novel results about the combinatorial properties of $(p,q)$--nanotubes, considering them as important examples of fullerenes and precursors to graphene, as illustrated in Figure \ref{fig:overview}. 
For this, we apply methods derived from spectral graph theory and probability theory.
Unless explicitly stated otherwise, when referring to fullerene, infinite $(p,q)$--nanotube, and graphene, we mean their corresponding graph representations henceforth.

\begin{figure}[H]
\begin{center}
\includegraphics[scale=0.8]{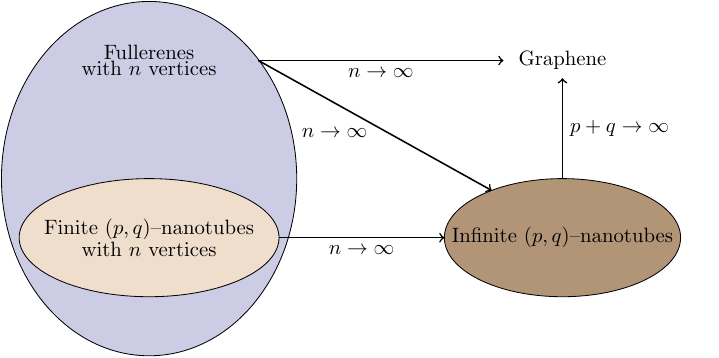}
\caption{Overview of links between finite fullerenes (left) and infinite hexagon arrangements (right).}\label{fig:overview}
\end{center}
\end{figure}

One key aspect of spectral graph theory is the study of a graphs' properties in relation to the eigenvalues of different matrices representing this graph.
The most common example is the adjacency matrix $A_n$ representing a graph $G_n$ with $n$ vertices, and its eigenvalues $\lbrace\lambda_j\rbrace_{1\leq j\leq n}$. 
The total number of closed paths of length $k$ starting at any vertex of $G_n$ is given as:
\begin{align*}
\tr \left(A_n^k\right) = \sum_{j=1}^n \lambda_j^k ,\quad k\in\N_0.
\end{align*}
As shown in \cite{Bille2020SpectralCO}, the shape of fullerenes, quantifiable via these power sums of different powers $k$ (known as \textit{Newton polynomials}) correlates significantly with their chemical and physical properties, notably their stability.

Let $\lbrace A_n\rbrace_n$ be a sequence of adjacency matrices, each representing a fullerene with $n$ vertices.
As the number of vertices $n$ tends towards infinity, the Newton polynomials $\tr\left(A_n^k\right)$ for any $k$ also approach infinity. 
However, normalizing by $n$ yields the average number of closed paths of length $k$ across all vertices of the graph:
\begin{align*}
\mu_k\left( A_n \right)
:= 
\frac{1}{n} \sum_{j=1}^n \lambda_j^k ,
\quad k\in\N_0.
\end{align*}
These counts can also be viewed as the $k^{th}$ moments of a discrete random variable $\Lambda_{G_n}$, referred to as the \textit{random eigenvalue} of $G_n$, with cumulative distribution function $F_{\Lambda_{G_n}}$:
\begin{align*}
\mu_k \left(A_n\right) 
= 
\mu_k\left(\Lambda_{G_n}\right) 
:= 
\E \Lambda_{G_n}^k = \int_\R x^k \diff F_{\Lambda_{G_n}}(x),\quad k\in\N_0.
\end{align*} 
Two natural questions arise: first, whether these moments $\mu_k\left( \Lambda_{G_n} \right)$ converge as $n$ increases, and if so, whether there exists an infinite graph $L$ whose average number of closed paths of length $k$ coincides with the limit of $\lbrace \mu_k\left(\Lambda_{G_n}\right) \rbrace_n$ as $n\to \infty$, for every $k\in\N_0$. 
The latter question can be restated as whether there exists a random eigenvalue $\Lambda_{L}$ such that $\Lambda_{G_n}$ converges weakly to $\Lambda_L$. 
Note that, since the vertex degrees of all the graphs and lattices considered here are bounded by $9$, the support of all random eigenvalues is compact due to Gershgorin's circle theorem. 
Consequently, and as a result of the Weierstrass approximation theorem, weak convergence $\Lambda_{G_n} \to \Lambda_L$ as $n\to \infty$ is equivalent to the convergence of moments $\mu_k\left(\Lambda_{G_n} \right) \to \mu_k\left(\Lambda_{L}\right)$ as $n \to \infty$ for all $k\in\N_0$.

Note that the above convergence is entirely determined by the choice of the sequence of fullerenes $\lbrace G_n \rbrace$, and this choice is not obvious at all.
In \cite[Conjecture 1]{bille2023random}, the authors state that the random eigenvalues of a sequence of uniformly randomly chosen fullerenes $\lbrace G_n \rbrace$ converge to the random eigenvalue of the hexagonal lattice, i.e., it holds
\begin{align*}
\mu_{2k} \left(\Lambda_{G_n}\right) 
\xrightarrow[n\to\infty]{}
\sum_{k_1+k_2+k_3=k}\binom{k}{k_1,k_2,k_3}^2,
\end{align*}
and $\mu_{2k+1}\left( \Lambda_{G_n}\right) = 0$ for all $k\in\N_0$.
Randomness in the choice of $G_n$ is crucial here, since one can easily construct a deterministic sequence of fullerenes whose moments $\mu_k$ either do not converge to the moments of the hexagonal lattice or do not converge at all.
 
As the number of vertices $n$ approaches infinity, the surface of fullerenes becomes a union of a vast number of hexagons (alongside twelve pentagons). 
Consequently, the primary focus shifts towards understanding the possible arrangements of these (infinitely) numerous hexagons, thereby revealing the structure of a large fullerene. 
The two extreme cases concerning the number of closed paths on infinitely large hexagon arrangements are as follows: 
$(a)$ The aforementioned hexagonal lattice, which has the smallest average number of closed paths, cf. \cite{bille2023random}; 
$(b)$ The $(p,q)$--nanotube with the smallest achievable \textit{circumference}, defined as $p+q$, and consequently, the greatest number of closed paths.

As presented in \cite{Grimme17}, the chemical stability of fullerenes is closely related to their shape, which is, in turn, determined by the location of the twelve pentagons on their surface. 
Loosely speaking, a more uniform distribution of pentagonal facets generally results in greater chemical stability. 
Therefore, the facet connectivity of fullerenes, particularly  $(p,q)$--nanotubes, is of special interest for chemists and physicists. 
Thus, from here on, we consider only the dual graphs of fullerenes, infinite $(p,q)$--nanotubes, and the hexagonal lattice. 
To emphasize this, we mark the corresponding variables related to these duals with a superscript ${}^*$. 

As shown in \cite{bille2023random}, it is easy to verify that
\begin{align*}
\mu_{2k}\left( A_n \right)
&= 
\mu_{k}\left(A_n^* + \frac{1}{2}D_n^*\right)\quad \text{and}\quad \mu_{2k+1}\left( A_n \right) = 0,\quad k\in\N_0,
\end{align*}
where $A_n$ is the adjacency matrix of a fullerene, and $A_n^*$ and $D_n^*$ denote the adjacency and degree matrix, respectively, of the corresponding dual fullerene. 
From a geometrical standpoint, this means adding $\sfrac{\text{deg}(v)}{2}$ loops to every vertex $v$ of the dual fullerene, which can also be viewed as a single loop with weight $\sfrac{\text{deg}(v)}{2}$.

This paper is structured as follows: 
first, we establish the formal framework needed to analyze the random eigenvalues $\Lambda_{p,q}^*$ of dual $(p,q)$--nanotubes, building on prior results concerning the hexagonal lattice and its dual. 

Section 3 presents the main results of this study, demonstrating the existence of limiting random eigenvalues for any given chiral vector $(p,q)$. 
It provides explicit functional representations of these random variables based on two independent uniform distributions and presents simpler expressions in case of zigzag and armchair nanotubes.

Next, we deduce three distinct formulas for the moments of $\Lambda_{p,q}^*$. 
The first formula is derived purely from combinatorial considerations, while the second and third formulas are obtained by computing a triple integral using techniques from real and complex analysis, respectively.  

We then present a general formula for the MGF of a random eigenvalue $\Lambda_{p,q}^*$ of an infinite dual $(p,q)$--nanotube, expressed as an integral over a discrete version of the modified Bessel function. 
Additionally, we demonstrate that the limit of these MGFs as $p+q\to \infty$ aligns with results about the hexagonal and triangular lattices presented in \cite{bille2023random}.

The paper concludes with explicit formulas for the PDFs of $\Lambda_{p,q}^*$ in case of zigzag ($q=0$) and armchair ($p=q$) nanotubes. 
For all other nanotubes, we propose a numerical algorithm to compute the PDFs and illustrate this algorithm using the dual infinite $(5,1)$--nanotube.

\section{Preliminaries}\label{sec:preliminaries}
In this section, we lay down the formal framework for our analysis by defining all objects under examination. 
To begin, we revisit some fundamental concepts from graph theory. 
We focus solely on undirected graphs denoted as $G_n=(V(G_n),E(G_n))=(V,E)$, i.e., a collection of vertices $V$ with $|V|=n$, and edges $E$, where each edge $(u,v)$ is indistinguishable from $(v,u)$ for all $u,v\in V$.

If the endpoints of an edge coincide, the edge is called a \textit{loop}. 
We also consider graphs with infinitely many vertices, each possessing finite degree, and arranged in such a way that they form a regular tiling. We refer to these locally finite graphs as \textit{(convex) lattices}. 
Their regularity ensures that the neighborhoods of a given size around any two vertices are isomorphic.

\subsection{Fullerenes, $\boldsymbol{(p,q)}$--nanotubes and graphene}
Recall some important properties of \textit{fullerenes}, defined as simple convex polytopes with $n$ vertices composed solely of pentagonal and hexagonal facets.
By $C_n$ we denote the set of all combinatorially non--equivalent fullerenes called \textit{isomers}. 
In other words, every element of $C_n$ is an equivalence class with respect to isomorphisms. 
For a more in-depth study of these polytopes, readers are referred to \cite{Bille2020SpectralCO,AndKarSkr16,BuchstEro}.

To enumerate all isomers within the set $C_n$, it is common to assume that the facets of each element in $C_n$ can be arranged in a so--called \textit{facet spiral sequence}. 
This sequence is a specific ordering of facets where each facet shares an edge with its neighbors and is represented by twelve integers indicating the positions of pentagons. Alternatively, the sequence can be expressed by a sequence of fives and sixes, where each number indicates whether the corresponding facet in the spiral is a pentagon or a hexagon.
A given fullerene may have multiple facet spirals, so the lexicographical smallest one is chosen to represent this fullerene. 
Here, the notation $C_{n,j}$ then refers to the $j^{\text{th}}$ isomer in $C_n$ according to the lexicographical order of its facet spiral sequence.
It is well known, however, that fullerenes which do not conform to this assumption exist (the first counterexample has $n=380$ vertices, see \cite{HouseofGraphs}).
To date, most studies analyzing specific fullerenes focus on examples with significantly fewer than $380$ vertices. 

Employing the so-called \textit{Schlegel projection}, which preserves vertex connectivity, fullerenes can be uniquely projected onto the two--dimensional plane, effectively treated as 3-regular, connected planar graphs. 
Here, equality and uniqueness of graphs is understood up to an isomorphism.

Additionally, fullerenes always contain twelve pentagons, regardless of the number of vertices $n\geq 20$, which must be even, $n\not= 22$, alongside with $\sfrac{n}{2}-10$ hexagons. 

For given non-negative integers $p,q$, we call a fullerene a \textit{finite $(p,q)$--nanotube}, if it features two cycles of edges, such that cutting the fullerene along these edges results in two distinct connected components, referred to as \textit{caps}, each containing six pentagons and a varying count of hexagons, along with a third component composed entirely of hexagonal facets, forming a tube. 
Figure \ref{fig:nanotube_decomposition} illustrates this decomposition for a finite $(5,0)$--nanotube.
The precise role of $p$ and $q$ in determining the boundary and shape of the tubular part is a technical matter detailed in \cite[Definition 2.7.2]{Buchstaber_2017}. 
For the purpose of this research, it is sufficient to understand the function of $p$ and $q$ provided later with the definition of the \textit{chiral vector}.

A dual fullerene, established in a broader setting as unique e.g. in \cite{whitney32}, forms a triangulation of the plane with vertices of degrees five or six. 
The set of dual fullerenes is of high theoretical interest, acting as a subset of \textit{convex tilings of the sphere}, as discussed in \cite{engel2018number}.

\begin{figure}[H]
\begin{center}
\includegraphics[scale=0.4]{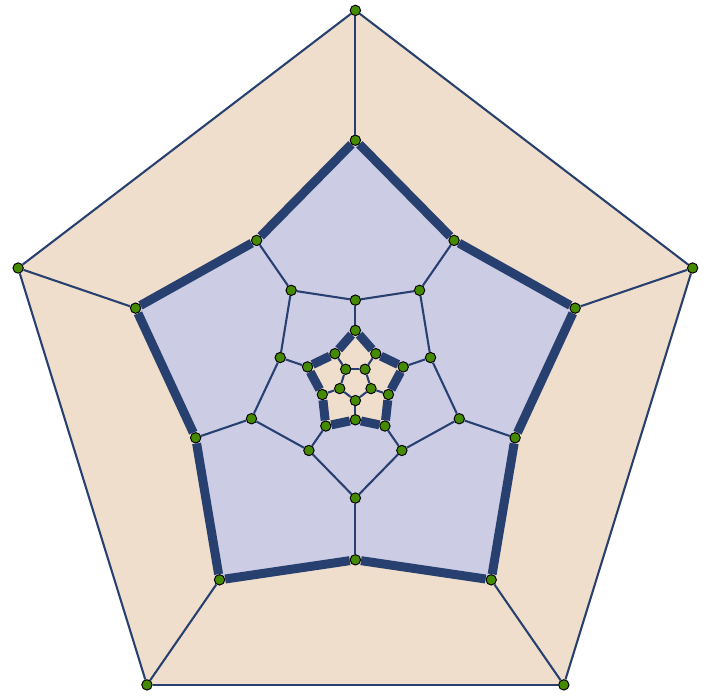}
\caption{Finite $(5,0)$--nanotube with $n=40$ vertices with bold lines denoting edges separating caps (beige facets) from the tubular part (violet facets). 
Note that the outer (uncolored) face represents the twelfth pentagon.}\label{fig:nanotube_decomposition}
\end{center}
\end{figure}
Before defining dual infinite $(p,q)$--nanotubes formally, let us briefly discuss an intuitive approach for constructing them starting with the hexagonal lattice.

Choose two hexagons in the hexagonal lattice and identify them, akin to rolling up a (infinitely large) piece of paper and gluing together two of its points.
A vector connecting the corresponding vertices of the two hexagons, illustrated in Figure \ref{fig:chiralvector}, is called the \textit{chiral vector}, and it entirely determines the structure of the resulting infinite \textit{$(p,q)$--nanotube}.
This chiral vector can be expressed as $p\cdot e_1 + q\cdot e_2$, where $p,q\in\Z$, and $e_1, e_2 \in\R^2$ are any two given linearly independent vectors in $\R^2$.
For simplicity and without loss of generality, we choose both vectors such that all edges in the resulting lattice have unit length, as depicted in Figure \ref{fig:chiralvector}.
Fixing the basis vectors $e_1$ and $e_2$, the vector of integers $(p,q)$ determines the chiral vector entirely, such that we refer to the \textit{chiral vector} as just the tuple $(p,q)$.

This construction aligns with the definition of finite $(p,q)$--nanotubes with $n$ vertices as $n\to \infty$ given before, wherein the caps remain unchanged, while the tubular segment extends infinitely. 
Consequently, both caps can be asymptotically neglected when calculating the average number of closed paths.
The tubular part of the nanotube is composed of infinitely many copies of \textit{hexagonal belts}, whose structure is determined by the aforementioned chiral vector. 
Figure \ref{fig:chiralvector} illustrates the connection between the chiral vector and the corresponding hexagonal belt.

\begin{figure}[H]
\begin{center}
\includegraphics[scale=1.3]{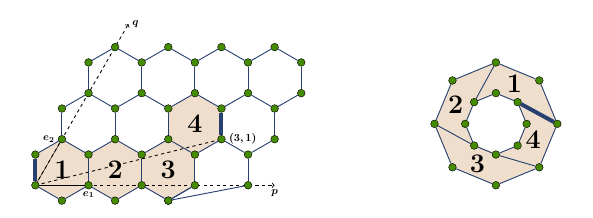}
\caption{Left: Part of a hexagonal lattice with two axes generated by the vectors $e_1 = \left(\sqrt{3}, 0\right)$ and $e_2 = \left(\sfrac{\sqrt{3}}{2}, \sfrac{3}{2}\right)$, along with the chiral vector $(3,1)$. Right: The shaded, numbered hexagons are cut out and glued together at the thicker edge, forming a belt.}\label{fig:chiralvector}
\end{center}
\end{figure} 

Due to the symmetry of the hexagonal facets, and consequently the entire hexagonal lattice, it is sufficient to consider only non--negative values for $p$ and $q$. 
Moreover, the chiral vector is symmetric, meaning a $(p,q)$--nanotube is identical to a $(q,p)$--nanotube. 
Both properties follow directly from \cite[Corollary 3.1]{DiCre}.
We further restrict our discussion to $(p,q)$--nanotubes with a circumference of at least 3, since no (feasible) belt with fewer than three hexagons can be constructed.
However, from a purely theoretical point of view, $p+q$ can be considered as any value starting from 1. 
The chiral vectors $(1,0)$ or $(2,0)$, $(1,1)$ also result in infinite nanotubes, but to ensure that the corresponding graph represents a three-dimensional polyhedron, i.e., that any two hexagons share no more than one common edge -- a necessary condition for  three-dimensional polyhedra according to Steinitz's theorem -- the circumference $p+q$ must be at least 3. 
This can be directly verified, as can the fact that any $(p,q)$ with $p+q\geq 3$ is suitable, cf. \cite{BuchEro17}. 
If we require that the infinite $(p,q)$--nanotube can be cut and considered as the tubular part of a finite $(p,q)$--nanotube, then $p+q$ must be at least 5. In the specific case of $p+q=5$, only $(5,0)$--nanotubes are possible. 
This follows from the fact that fullerenes do not have 3- and 4-belts of hexagons, and nontrivial 5-belts are found only on $(5,0)$--nanotubes. 
For a more detailed explanation the reader is referred to \cite{BuchEro17}.
In the following, we consider chiral vectors $(p,q)\in \mathbf{N}:=\lbrace (p,q)\in\N_0^2~|~ p+q \geq 5,~ q\leq p \rbrace $. Now we can formulate
\begin{definition}\label{def:H_T_N_N*}
Let $(p,q)\in \mathbf{N}$ be given. 
We call the graph
\begin{enumerate}[(i)]
\item $H$ a \textit{hexagonal lattice} if its set of vertices is given by
\begin{align*}
V(H):=\left\lbrace \left(\sqrt{3} x+\frac{y\sqrt{3}}{2},\frac{3y}{2} +c\right)^{\intercal}\in\R^2 ~\middle|~ x,y\in\Z,  ~c\in\lbrace 0,1 \rbrace \right \rbrace,
\end{align*}
and every vertex is connected with its three nearest neighbours (w.r.t. the Euclidean distance) by an edge.

\item $T^*$ a \textit{triangular lattice}  if its set of vertices is given by
\begin{align*}
V\left(T^*\right):=\left\lbrace v\in V(H)~\middle|~ c=0 \right \rbrace
\end{align*}

and every vertex is connected with its six nearest neighbours (w.r.t. the Euclidean distance) by an edge. Additionally, we require that every vertex of $T^*$ has three loops, or equivalently one loop of weight 3.

\item $N_{(p,q)}^*$ a \textit{dual infinite $(p,q)$--nanotube} if $N_{(p,q)}^*$ has the same sets of vertices and edges as $T^*$ with the condition that every vertex $v\in V(N_{(p,q)}^*)$ is identified with $v+j\sqrt{3}(p,q)^{\intercal}$ for every $j\in\Z$. 
According to its chiral vector $(p,q)$, every dual infinite $(p,q)$--nanotube can be assigned to exactly one of the three groups:
\begin{enumerate}[(i)]
\item If $q=0$, then the nanotube is called a \textit{zigzag nanotube}.
\item If $p=q$, then the nanotube is called an \textit{armchair nanotube}.
\item Otherwise, i.e., if $0<q<p$, the nanotube is called a \textit{chiral nanotube.}
\end{enumerate}
\end{enumerate}
\end{definition}

%---------------------------------------------------------------------------------------------------------

\section{Spectrum of dual $\boldsymbol{(p,q)}$-nanotubes}  \label{sect:RVLambda}

In this section, we first derive a representation of random eigenvalues $\Lambda_{p,q}^*$ for general dual $(p,q)$--nanotubes. 
We also present two more practical formulas for these random eigenvalues in the specific cases of zigzag ($q=0$) and armchair ($p=q$) nanotubes.

Finally, we analyze the asymptotic behavior of $\Lambda_{p,q}^*$ as $p+q\rightarrow\infty$, demonstrating that the random eigenvalues of nanotubes converge to the random eigenvalues of the triangular lattice.  

%-----------------------------------------------------------------------------------------------------------

\subsection{Distribution of random eigenvalues $\boldsymbol{\Lambda_{p,q}^}*$} \label{sect:RVLambda_subsect:RV}
\noindent 
For $(p,q)\in\mathbf{N}$ and $k\in\N_0$, the number of closed paths with $k$ steps on a dual infinite $(p,q)$--nanotube given in \cite{Cotfas00} is
\begin{align}\label{eq:fCotfas}
\mu_k\left( \Lambda_{p,q}^* \right) 
&= 
\frac{1}{(2\pi)^3} \sum_{j\in\Z} \int_{-\pi}^\pi\int_{-\pi}^\pi\int_{-\pi}^\pi \left| e^{ix} + e^{iy} + e^{iz} \right|^{2k} e^{ij\left(p\left( x-z\right) +q\left(y-z\right)\right)} \diff x \diff y \diff z .
\end{align} 
 
\begin{theorem}\label{thm:RV_pq}
For all $(p,q)\in\mathbf{N}$, it holds
\begin{align} 
\Lambda^*_{p,q} \eqd
3+2\left( 
\cos \mathcal{U} + 
\cos \left( \frac{p \, \mathcal{U} + 2\pi \mathcal{J}_{p,q}}{p+q} \right) + 
\cos \left(\frac{q \, \mathcal{U} - 2\pi \mathcal{J}_{p,q}}{p+q} \right) \right), \label{theorem1:1}
\end{align}
where $\mathcal{U}\sim U(0,\pi)$ and
$\mathcal{J}_{p,q}\sim U\lbrace 0,\ldots, p+q-1 \rbrace$ are independent uniformly distributed random variables.
\end{theorem}

\begin{proof}
Denote by $\delta_0$ the Dirac delta function. 
Applying the \textit{Poisson summation formula} \cite[p. 82]{Vlad71}
\begin{align*}
\sum_{j\in\Z} e^{ i j \alpha} = 2\pi \sum_{j\in\Z} \delta_0(\alpha - 2\pi j),\quad \alpha\in\C,
\end{align*}
with $\alpha={p(x-z) + q(y-z)}$ to relation \eqref{eq:fCotfas} we get
\begin{align}
\mu_k\left( \Lambda_{p,q}^* \right) 
&= 
\frac{1}{(2\pi)^2} 
\sum_{j\in\Z}
\int_{-\pi}^\pi  \int_{-\pi}^\pi  \int_{-\pi}^\pi  
	\left| e^{ix} + e^{iy} + e^{iz} \right|^{2k} 
	\delta_0\left(p\left( x-z\right) +q\left(y-z\right) - 2\pi j\right) 
\diff z \diff x \diff y.
\label{formula: mu_k delta_0}
\end{align}
For the inner integral we get by substitution and using the property 
$\int_{-\infty}^{\infty}
f(u) \delta_0(u)
\diff u = f(0)$
that
\begin{align*}
&\int_{-\pi}^\pi  
	\left| e^{ix} + e^{iy} + e^{iz} \right|^{2k} 
	\delta_0\left(p\left( x-z\right) +q\left(y-z\right) - 2\pi j\right) 
\diff z
%=
%&\frac{1}{p+q}\int_{p(x-\pi) + q(y-\pi) - 2\pi j}^{p(x+\pi) + q(y+\pi) - 2\pi j} 
%\left|e^{ix} + e^{iy} + e^{i\frac{px + qy -2\pi j -u}{p+q}} \right|^{2k} \delta_0(u) 
%\diff u\\
=
\frac{1}{p+q}
\left| e^{ix} + e^{iy} + e^{i\frac{px +qy - 2\pi j}{p+q}} \right|^{2k}.
\end{align*}
Note that as $z\in [-\pi,\pi)$ and $j\in\Z$, it must hold
\begin{align*}
\frac{px + qy-2\pi j}{p+q} \in [-\pi,\pi) 
\quad &\Leftrightarrow \quad
%j \in \left[\frac{p(x-\pi)+q(y-\pi)}{2\pi}, ~ \frac{p(x+\pi)+q(y+\pi)}{2\pi} \right) \cap \Z \\
j - \frac{p(x-\pi) + q(y-\pi)}{2\pi} \in (0,p+q] \cap \Z.
\end{align*}
Thus, the sum in \eqref{formula: mu_k delta_0} is finite, and we get
\begin{align*}
\mu_k\left( \Lambda_{p,q}^* \right) 
&= 
\frac{1}{(p+q)(2\pi)^2}  
\int_{-\pi}^\pi \int_{-\pi}^\pi 
\sum_{j = 1}^{p+q}
	\left| e^{ix} + e^{iy} + e^{i\frac{px+qy-2\pi j}{p+q}} \right|^{2k}
	\diff x \diff y\\
&= 
\frac{1}{(p+q)(2\pi)^2} 
\sum_{j=1}^{p+q}  
\int_{-\pi}^\pi \int_{-\pi}^\pi  
	\underbrace{\left| e^{ix} \right|^{2k}}_{=1} \cdot 
	\left| 1 + e^{i(y-x)} + e^{i \frac{q(y-x) - 2\pi j}{p+q} } \right|^{2k} 
	\diff x \diff y\\
&\overset{\theta:=y-x}{=} 
\frac{1}{(p+q) 2\pi} 
\sum_{j=1}^{p+q} 
\int_{-\pi}^\pi  
	\left|1+e^{i\theta}+e^{i\frac{q\theta - 2\pi j}{p+q}} \right|^{2k} 
	\diff \theta
= 
\E \mathcal{X}_{p,q}^k,
\end{align*}
where $\mathcal{X}_{p,q} := \left|1+e^{iU}+e^{i \frac{qU - 2\pi J}{p+q}} \right|^{2}$ is a random variable constructed from two independent uniformly distributed random variables 
$\mathcal{U}\sim U(-\pi,\pi)$ and
$\mathcal{J}_{p,q}\sim U\lbrace 1,\ldots, p+q \rbrace$.
Then, due to Euler's formula and some basic trigonometric transformations, $\mathcal{X}_{p,q}$ can be expressed as
\begin{align}
\mathcal{X}_{p,q}
&=
3 + 2\left( 
\cos \mathcal{U} + 
\cos \left( \frac{p \,\mathcal{U} + 2\pi \mathcal{J}_{p,q}}{p+q} \right) + 
\cos \left(\frac{q \, \mathcal{U} - 2\pi \mathcal{J}_{p,q}}{p+q}\right) 
\right).  \label{formula:3cos}
\end{align}
For the sake of readability, and since $\cos(\cdot)$ is $2\pi$--periodic, we can shift the support of $\mathcal{J}_{p,q}$ to $\lbrace0,\ldots, p+q-1\rbrace$.

We can restrict the support of $\mathcal{U}$ to $(0,\pi)$ since the distribution of \eqref{formula:3cos} is even with respect to $\mathcal{U}$, i.e., it holds
\begin{align}
\mathcal{X}_{p,q}
\eqd
3+ 2 \left(
\cos \left|\mathcal{U}\right| + 
\cos \left( \frac{p \left|\mathcal{U}\right| + 2\pi \mathcal{J}_{p,q}}{p+q} \right) + 
\cos \left(\frac{q \left|\mathcal{U}\right| - 2\pi \mathcal{J}_{p,q}}{p+q}\right)
\right),
\label{formula_3cos_2}
\end{align}
with $\mathcal{U}\sim U(-\pi,\pi)$.
To see this, recall that $\cos x$ is even in $x$, and that any linear transformation of an even function remains even. 
Hence, we can disregard the linear transformation and the first cosine on the right hand-side of \eqref{formula_3cos_2}.
It remains to prove the symmetry of the distribution with respect to $\mathcal{U}$ for the function
\begin{align*}
h(\mathcal{U},\mathcal{J}_{p,q},p,q) :=
\cos \left( \frac{p\,\mathcal{U}  +2\pi \mathcal{J}_{p,q}}{p+q} \right) + \cos \left(\frac{q\,\mathcal{U} - 2\pi \mathcal{J}_{p,q}}{p+q}\right),
\end{align*}
using the law of total probability over all possible values of $\mathcal{J}_{p,q}$.
It is straightforward to verify that $h(\mathcal{U},0,p,q)$ is even for all $(p,q)\in\mathbf{N}$, as well as $h\left(\mathcal{U},\frac{p+q}{2},p,q\right)$ when $p+q$ is even.
For all other values of $\mathcal{J}_{p,q}$, we get that
\newpage
\begin{align*}
&h\left(\mathcal{U}, \mathcal{J}_{p,q},p,q\right)
+ 
h\left(\mathcal{U}, p + q - \mathcal{J}_{p,q},p,q\right)\\
=~
&\cos\left(\frac{p\,\mathcal{U} + 2\pi \mathcal{J}_{p,q}}{p+q}\right) + 
\cos\left( \frac{q\, \mathcal{U} - 2\pi \mathcal{J}_{p,q}}{p+q}\right) + 
\cos\left(\frac{p\, \mathcal{U} - 2\pi \mathcal{J}_{p,q}}{p+q}\right) + 
\cos\left( \frac{q\, \mathcal{U} + 2\pi \mathcal{J}_{p,q}}{p+q}\right)
\end{align*}
is also even. 
This, finally, proves the claim.
\end{proof}

%%%%%%%%%%%%%%%%%%%%%%%%%%%%%%%%%%%%%%%%%%%%%%%%%%%%%%%%%%%%%

\begin{remark}
\begin{enumerate}[(i)]
\item For chiral nanotubes $(0<q<p)$, we can rewrite the distribution of $\Lambda_{p,q}^*$ in terms of \textit{Chebyshev polynomials of the first kind}, which are defined as the unique polynomials $T_n$ of degree $n\in\N_0$ satisfying 
\begin{align*}
T_n(\cos\theta ) = \cos \left( n\theta\right),\quad \theta\in\R.
\end{align*}
The multiple-angle formula for cosine yields
\begin{align*}
\Lambda_{p,q}^* 
\overset{d}{=} 
1 + 8T_p(\mathcal{U}) T_q(\mathcal{U}) T_{p+q}(\mathcal{U}).
\end{align*}

\item For the two important cases of zigzag and armchair nanotubes, one way to derive important properties of their random eigenvalues is to express the righthand-side of \eqref{theorem1:1} in terms of $\cos\left(\sfrac{\mathcal{U}}{2}\right)$. 
It holds
\begin{align*}
&\cos \mathcal{U} + \cos \left( \frac{p\, \mathcal{U} + 2\pi \mathcal{J}_{p,q}}{p+q} \right) + \cos \left(\frac{q\, \mathcal{U} - 2\pi \mathcal{J}_{p,q}}{p+q} \right)
 =\\ 
&2	\cos^2\left(\frac{\mathcal{U}}{2}\right) + 
2	\cos\left(\frac{\mathcal{U}}{2}\right) 
 	\cos\left(\frac{(p-q)\,\mathcal{U}+4\pi \mathcal{J}_{p,q}}{2(p+q)}\right) - 1,
\end{align*} 
and 
\begin{align*} 
\cos\left(\frac{(p-q)\,\mathcal{U} + 4\pi \mathcal{J}_{p,q}}{2(p+q)}\right)
&= 
\begin{cases}
\sqrt{1-\cos^2\left(\frac{\mathcal{U}}{2}\right)} \sin\left(\frac{2\pi \mathcal{J}_{p,0}}{p}\right),\quad &\text{ if } q=0,\\
\cos\left(\frac{\pi \mathcal{J}_{p,p}}{p}\right),\quad &\text{ if } p=q.
\end{cases}
\end{align*}
Furthermore, we get easily the following density transformation:
\begin{align}
f_{\cos \frac{\mathcal{U}}{2}}(x)
= 
\frac{2}{\pi\sqrt{1-x^2}} \mathbbm{1}\lbrace 0 \leq x < 1 \rbrace, \label{density_U2}
\end{align}
which can be thought of as the (normalized) right half of the arcsine distribution on the interval $(-1,1)$.
\end{enumerate}
%%%%%%%%%%%%%%%%%%%%%%%   Additional notes   %%%%%%%%%%%%%%%%%%%%%%%%%%%%%%%%%%%%%%

%Further, using the multiple--angle formula for cosine one gets
%\begin{align*}
%&1 + 8 \cos\left( (p+q) U \right)\cos\left(pU\right)\cos\left(qU\right)\\
%=
%&1 + 8 \sum_{k_1=0}^{\lfloor \frac{p+q}{2} \rfloor} \sum_{k_2=0}^{\lfloor \frac{p}{2} \rfloor} \sum_{k_3=0}^{\lfloor \frac{q}{2} \rfloor}
%(-1)^{k_1+k_2+k_3} \binom{p+q}{2k_1} \binom{p}{2k_2} \binom{q}{2k_3}
%\sin^{2(k_1+k_2+k_3)}(U) \cos^{2(p+q-k_1-k_2-k_3)}(U)\\
%=
%&1 + 8 \cos^{2(p+q)}(U)\sum_{k_1=0}^{\lfloor \frac{p+q}{2} \rfloor} \sum_{k_2=0}^{\lfloor \frac{p}{2} \rfloor} \sum_{k_3=0}^{\lfloor \frac{q}{2} \rfloor}
%(-1)^{k_1+k_2+k_3} \binom{p+q}{2k_1} \binom{p}{2k_2} \binom{q}{2k_3}
%\left(-\tan^2(U)\right)^{2(k_1+k_2+k_3)}
%\end{align*}
%%%%%%%%%%%%%%%%%%%%%%%%%%%%%%%%%%%%%%%%%%%%%%%%%%%%%%%%%%%%%
%Moreover, for chiral nanotubes $(0<q<p)$, it holds
%\begin{align*}
%\Lambda_{p,q}^* 
%&\overset{d}{=} 
%3+2\left( \cos U_{p,q} + \cos \left( \frac{pU_{p,q}}{p+q}\right) + \cos \left(\frac{qU_{p,q}}{p+q}\right) \right)\\
%&\overset{d}{=} 
%1 + 8 \cos\left(pU\right)\cos\left(qU\right)\cos\left( (p+q) U \right)
%\end{align*}
%where $U\sim U(-\pi,\pi)$, $U_{p,q}\sim(-(p+q)\pi,(p+q)\pi)$.

\end{remark}

The latter remark leads directly to
\begin{corollary}\label{corollary:distribution_Lambda}
It holds
\begin{align*}
\Lambda_{p,0}^* 
&\overset{d}{=} 
4\left(1+\cos\left(\frac{2\pi \mathcal{J}_{p,0}}{p} \right)\right) \mathcal{V}^2 - 4\sin\left(\frac{2\pi \mathcal{J}_{p,0}}{p,0}\right) \mathcal{V}\sqrt{1-\mathcal{V}^2} + 1 ,\\
\Lambda_{p,p}^* 
&\overset{d}{=}  
4\mathcal{V}^2 + 4\cos\left(\frac{\pi \mathcal{J}_{p,p}}{p}\right) \mathcal{V} + 1,
\end{align*}
with independent random variables $\mathcal{V}$, $\mathcal{J}_{p,0}\sim U\lbrace 0,\ldots,p-1\rbrace$, $\mathcal{J}_{p,p}\sim U\lbrace 0, \ldots, 2p-1\rbrace$, where the PDF of $\mathcal{V}$ is given in \eqref{density_U2}.
\end{corollary}

\color{black}
%%%%%%%%%%%%%%%%%%%%%%%%%%%%%%%%%%%%%%%%%%%%%%%%%%%%%%%%%%%%%
Additionally, Theorem \ref{thm:RV_pq} provides a theoretical explanation for the intuitive idea that a dual infinite $(p,q)$--nanotube with increasing circumference converges locally to the triangular lattice, in the sense that  $\Lambda_{p,q}^*$ converges weakly to the random variable $\mathcal{T}$ (denoting the random eigenvalue of $T^*$) whose PDF is presented in \cite[Proposition 4]{bille2023random} as
\begin{align*}
f_{\mathcal{T}}(x) 
=
\frac{1}{2}\int_0^\infty t J_0\left( t\sqrt{x}\right) J_0^3(t) \diff t,\quad x>0,
\end{align*}
where $J_n$ denotes the \textit{Bessel function of the first kind of order $n\in\Z$} defined as 
\begin{align*}
J_n(x)
:=
\sum_{k=0}^\infty \frac{(-1)^k}{k!(k+n)!} \left(\frac{x}{2} \right)^{2k+n}, \quad x\in\C.
\end{align*}

\begin{proposition}\label{proposition:weak_convergence_Lambdapq}
As $p+q\rightarrow \infty$, such that $\frac{p}{p+q}\to c \in[0,1]$, the following weak convergence holds for the random eigenvalue of a dual infinite $(p,q)$--nanotube:
\begin{align*}
\Lambda^*_{p,q} \xrightarrow[p+q\rightarrow\infty]{d} \mathcal{T}.
\end{align*}
\end{proposition}

\begin{proof}
Let $\mathcal{U},\mathcal{V} \sim U(0,2\pi)$ be two independent random variables.
By Theorem \ref{thm:RV_pq}, we directly infer:
\begin{align*} 
\Lambda^*_{p,q} 
\xrightarrow[p+q\rightarrow\infty]{d} 
3 + 2\left( \cos \mathcal{U} + \cos \left( (1-c)\,\mathcal{U} + \mathcal{V} \right) + \cos \left( c\, \mathcal{U} - \mathcal{V}\right) \right)
=: \mathcal{X}_c,
\end{align*}
as $\frac{p}{p+q}\to c $.
It remains to prove that the distribution on the right-hand side does not depend on $c$ and coincides with the distribution of $\mathcal{T}$. 
Referring to \cite[Remark 1]{bille2023random}, it holds for every $c\in[0,1]$, that
\begin{align*}
\left| 1 + e^{i\, \mathcal{U}} + e^{i\left(c\, \mathcal{U} - \mathcal{V} \right)} \right|^2 
&= 
\left( 1+ \cos \mathcal{U} + \cos\left( c\, \mathcal{U} - \mathcal{V} \right) \right)^2 + \left( \sin \mathcal{U} + \sin \left( c\, \mathcal{U} - \mathcal{V} \right) \right)^2\\
&= 
1 + \underbrace{\cos^2 \mathcal{U} +\sin^2 \mathcal{U}}_{=1} + \underbrace{\cos^2\left(c\, \mathcal{U} - \mathcal{V} \right) + \sin^2\left(c \,\mathcal{U} - \mathcal{V} \right)}_{=1} \\
&+ 2\big( \cos \mathcal{U} + \cos\left( c\, \mathcal{U} - \mathcal{V}\right) +\underbrace{\cos \mathcal{U} \cos\left( c\, \mathcal{U} - \mathcal{V} \right) + 
\sin \mathcal{U} \sin\left( c\, \mathcal{U} - \mathcal{V}\right)}_{=\cos\left(\mathcal{U} - c\, \mathcal{U} + \mathcal{V}\right)}\big)
\\
&=
3 + 2\left(\cos\left( \mathcal{U}\right) + \cos\left( \left(1-c \right)\,\mathcal{U} + \mathcal{V} \right) + \cos\left( c\, \mathcal{U} - \mathcal{V} \right)\right)
=
\mathcal{X}_c.
\end{align*}
Next, employing the proof of Proposition \ref{prop:moments1} with $j=0$, one gets
\begin{align*}
\E \mathcal{X}_c^k 
= 
\frac{1}{\left(2\pi\right)^2}\int_{-\pi}^\pi \int_{-\pi}^\pi \left| 1 + e^{ix}+e^{i\left(c x -y\right)} \right|^{2k} \diff x \diff y 
=
\sum_{k_1+k_2+k_3=k}\binom{k}{k_1,k_2,k_3}^2=\E \mathcal{T}^k,
\end{align*}
where the last equality is presented in \cite[Proposition 3]{bille2023random}.
Since both $\mathcal{X}_c$ and $\mathcal{T}$ have compact support, the equality of moments implies the equality in distribution. 
\end{proof}

%--------------------------------------------------------------------------------------------------------------------------------------------------------------------------------------------------

\subsection{Moments  of $\boldsymbol{\Lambda^*_{p,q}}$}\label{sect:RVLambda_subsect:Moments}

Due to their representations as sums involving binomial and multinomial coefficients, moments of random eigenvalues of (dual) infinite $(p,q)$--nanotubes (as well as those of the hexagonal and triangular lattices) are special instances of more general mathematical and combinatorial problems. 
Notably, these problems include topics like \textit{counting Abelian squares} \cite{counting_abelian_squares}. 
In particular, a broader and more applicable introduction on hypergeometric identities can be found in \cite{petkovsek1996b}. 
In this book, Petkovsek et al. introduce and discuss the necessity for various problems of that kind, as well as provide an extended overview of the literature.

In this section, we present three different formulas for the moments \eqref{eq:fCotfas}. 
The proof of formula \eqref{eq:moments1} utilizes basic calculus of real integrals, while the derivation of representation \eqref{eq:moments2} employs complex analysis. 
The last formula \eqref{eq:moments3} stems from a combinatorial application of the results from \cite{DiCre} based on probability generating functions.

\begin{proposition}\label{prop:moments1}
For $(p,q)\in\mathbf{N}$ and $k\in\N$, the $k^{\text{th}}$ moment of $\Lambda_{p,q}^*$ is given by
\begin{align}
\mu_k\left(\Lambda_{p,q}^*\right)
&= 
\sum_{j\in\Z} \sum_{k_1+k_2+k_3=k} \sum_{k_1'+k_2'+k_3' = k} \binom{k}{k_1,k_2,k_3} \binom{k}{k_1',k_2',k_3'}
\mathbbm{1} \begin{Bmatrix*}[l] k_1' = k_1 + jp\\ k_2' = k_2 + jq \\ k_3' = k_3 - j(p+q) \  \end{Bmatrix*} . \label{eq:moments1}
\end{align}
\end{proposition}

\begin{proof}
Using formula \eqref{eq:fCotfas}, it follows by simple calculations that 
\begin{align*}
\mu_k\left(\Lambda_{p,q}^*\right)
&= 
\frac{1}{(2\pi)^3} \sum_{j\in\Z} \int_{-\pi}^\pi \int_{-\pi}^\pi \int_{-\pi}^\pi
\left| e^{ix} + e^{iy} + e^{iz} \right|^{2k} e^{ij\left(p (x-z) + q(y-z) \right)}
\diff x\diff y\diff z \\
&=
\frac{1}{(2\pi)^3} \sum_{j\in\Z} \int_{-\pi}^\pi \int_{-\pi}^\pi \int_{-\pi}^\pi
\left( \sum_{k_1+k_2+k_3=k} \binom{k}{k_1,k_2,k_3} e^{i\left( k_1x + k_2y + k_3z \right)} \right)\\
 & \hspace*{4cm}\times\left( \sum_{k_1'+k_2'+k_3'=k} \binom{k}{k_1',k_2',k_3'} e^{-i\left( k_1'x + k_2'y + k_3'z \right)} \right)
 e^{ij\left(p(x-z)+ q(y-z)\right)}
\diff x\diff y\diff z\\
&=
\frac{1}{(2\pi)^3} \sum_{j\in\Z} \sum_{k_1+k_2+k_3=k} \sum_{k_1'+k_2'+k_3'=k} \binom{k}{k_1,k_2,k_3} \binom{k}{k_1',k_2',k_3'} \\
&\hspace*{4cm} \times \int_{-\pi}^\pi e^{ix\left(k_1 - k_1' + jp\right)} \diff x 
\int_{-\pi}^\pi e^{iy\left(k_2 - k_2' + jq\right)} \diff y 
\int_{-\pi}^\pi e^{iz\left(k_3 - k_3' - j(p+q)\right)}  \diff z \\
&= 
\sum_{j\in\Z} \sum_{k_1+k_2+k_3=k} \sum_{k_1'+k_2'+k_3' = k} \binom{k}{k_1,k_2,k_3} \binom{k}{k_1',k_2',k_3'}
\mathbbm{1} \begin{Bmatrix*}[l] k_1' = k_1 + jp\\ k_2' = k_2 + jq \\ k_3' = k_3 - j(p+q) \end{Bmatrix*} .
\end{align*}
\end{proof}

\begin{remark}\label{remark:number of closed walks}
\begin{enumerate}[(i)]
\item Based on the idea of an infinite nanotube as a rolled-up hexagonal lattice, and by summing  the number of walks between specific vertices on the hexagonal lattice, as detailed in \cite[Proposition 2]{DiCre}, we can derive an alternative expression for the number of closed walks on dual infinite $(p,q)$--nanotubes. 
Specifically, for $k\in\N$ we have
\begin{align}
\mu_k\left(\Lambda_{p,q}^*\right)
&=
\sum_{k_1+k_2+k_3=k} 
\binom{k}{k_1,k_2,k_3}^2 
\left( 1+2\sum_{l=1}^{\left\lfloor \frac{k_1}{p+q} \right\rfloor} 
\frac{\binom{2k_1-lq}{k_1+lp} 
\binom{k}{k_1-lq}}{\binom{k}{k_1} 
\binom{2(k-k_1)}{k-k_1}} 
\right).\label{eq:moments2}
\end{align}
In particular, for $p=5,q=0$ we get 
\begin{align*}
\mu_k\left(\Lambda^*_{5,0}\right)
&=
\sum_{k_1+k_2+k_3=k} \binom{k}{k_1,k_2,k_3}^2 \left( 1+f(k_1,k)\right),
\end{align*}
where
\begin{equation*}
f(k_1,k) :=
\frac{2((k-k_1)!)^2}{(2(k-k_1))!} \binom{2k_1}{k_1+5}{}_6F_5 \left(
\begin{matrix}
1,\frac{5-k_1}{5},\frac{6-k_1}{5},\frac{7-k_1}{5},\frac{8-k_1}{5},\frac{9-k_1}{5} \\
\frac{6+k_1}{5},\frac{7+k_1}{5} ,\frac{8+k_1}{5},\frac{9+k_1}{5},\frac{10+k_1}{5}  
\end{matrix};
-1
\right),
\end{equation*}
and ${}_6F_5$ is the \textit{generalized hypergeometric function} defined for $p,q\in\N$ as
\begin{align*}
{}_qF_p \left(\begin{matrix}
a_1,\ldots,a_p\\
b_1,\ldots,b_q
\end{matrix}; x\right)
:=
\sum_{j=0}^{\infty} \frac{\left(a_1\right)_k\cdot \ldots \cdot \left(a_p\right)_k}{\left(b_1\right)_k\cdot \ldots \cdot \left(b_q\right)_k} \cdot \frac{x^k}{k!},\quad a_1,\ldots,a_p,b_1,\ldots,b_q,x\in \R,
\end{align*}
$(z)_k$ being the rising factorial. 

Alternatively, for $k\in\N_0$, it holds
\begin{equation}
\mu_k\left( \Lambda^{*}_{p,q} \right) 
= 
\sum_{k_1+\ldots+k_7=k}
\binom{k}{k_1,\ldots,k_7} 3^{k_1} 
\mathbbm{1}
\left\{
\begin{array}{l}
pk-k_2-k_4+k_5+k_7 \equiv 0 \text{ mod }p\\
p\left(-k_2 + k_3 + k_5 - k_6\right) = q\left( k_2 + k_4 - k_5 - k_7 \right)
\end{array}
\right\}. 
\label{eq:moments3}
\end{equation}
Detailed proofs for both formulas \eqref{eq:moments2} and \eqref{eq:moments3} can be found in the Appendix.

\item It is straightforward to see that the inner sum of the right-hand side of \eqref{eq:moments2} is empty as $p+q\to\infty$, hence it implies the same weak convergence proven in Proposition \ref{proposition:weak_convergence_Lambdapq}. 
For $k\in\N_0$, it holds
\begin{align}
\mu_k\left(\Lambda_{p,q}^*\right) \xrightarrow[p+q\to\infty]{}\mu_k(\mathcal{T}) = \sum_{k_1+k_2+k_3=k}\binom{k}{k_1,k_2,k_3}^2. \label{eq:moment_convergence}
\end{align} 
Thus, expressions \eqref{eq:moments1} and \eqref{eq:moments3} also converge to the right-hand side of \eqref{eq:moment_convergence}.
\end{enumerate}
\end{remark}

\subsection{Moment generating function of $\Lambda_{p,q}^*$}

In this section, we derive formulas for the MGF 
$m_{\Lambda_{p,q}^*}(t) = \E \exp\left(t\Lambda_{p,q}^*\right)$, $t\in\R$,
of $\Lambda_{p,q}^*$ and its limiting behavior as $p+q$ tends to infinity. 
This involves the \textit{modified (or hyperbolic) Bessel functions of the first kind of order $n$}, $n\in\Z$, which can be defined based on the \textit{Bessel functions of the first kind of order $n$} as 
\begin{align*}
I_n(x)
:=
i^{-n}J_n(ix)
=
\sum_{k=0}^\infty \frac{1}{k! (k+n)!} \left( \frac{x}{2} \right)^{2k+n},\quad x\in\C.
\end{align*}
Notably, we focus on the particular case $n=0$, as it is of special use in the subsequent discussion.

Additionally, we need a modification of the following general integral representation of $I_0$, which can be found, e.g., in \cite{abra}:
\begin{align}
I_n(x) 
= 
\frac{1}{\pi} \int_0^\pi e^{x\cos(\varphi)}\cos(n\varphi) \diff \varphi. \label{eq:I0_integral_general}
\end{align}
By making a suitable substitution and applying formulas for the sum of angles in $\cos(\cdot)$ and $\sin(\cdot)$, Equation \eqref{eq:I0_integral_general} can be rewritten as
\begin{align}
I_0\left(\sqrt{\alpha^2+\beta^2}\right) 
= 
\frac{1}{2\pi} \int_0^{2\pi} \exp\left(\alpha\cos \varphi + \beta \sin \varphi\right)
\diff \varphi,\quad \alpha,\beta\in\R,\label{eq:I0_integral}
\end{align}
cf. Proposition \ref{prop:I0_integral} in Appendix.

\begin{theorem}\label{theorem:mgf}
For $(p,q)\in\mathbf{N}$, it holds
\begin{align*}
m_{\Lambda_{p,q}^*}(t) 
= 
\frac{e^{3t}}{\pi} \int_{0}^{\pi} 
e^{2t\cos\theta} \widehat{I}_{0}\left(4t\alpha_{p,q}(\theta), 4t\beta_{p,q}(\theta),p,q \right)
\diff \theta,\quad t\in\R,
\end{align*}
where 
\begin{align}
\widehat{I}_{0}\left(\alpha,\beta, p,q\right) 
&= 
\frac{1}{p+q}\sum_{j=0}^{p+q-1} \exp\left(
\alpha \cos\left( \frac{2\pi j}{p+q}\right)+
\beta \sin \left(\frac{2\pi j}{p+q}\right)
\right),\label{def:Kpq}\\
\alpha_{p,q}(\theta) 
&= 
\cos\left(\frac{\theta}{2}\right)\cos \left( \frac{\theta(p-q)}{2(p+q)}\right),\quad 
\beta_{p,q}(\theta) 
= 
\cos\left(\frac{\theta}{2}\right) \sin\left(\frac{\theta(p-q)}{2(p+q)}\right)\label{def:alphabeta}.
\end{align}
\end{theorem}

\begin{proof}
For $t\in\R$, we have
\begin{align*}
m_{\Lambda_{p,q}^*}(t) 
&:= 
\E e^{t\Lambda_{p,q}^*} 
\overset{\eqref{theorem1:1}}{=}
 e^{3t}~ \E \exp\left( 
 2t\left(
 \cos \mathcal{U} + \cos\left(\frac{p \mathcal{U} + 2\pi \mathcal{J}_{p,q}}{p+q} \right) + \cos\left( \frac{q \mathcal{U}-2\pi \mathcal{J}_{p,q}}{p+q}\right)
 \right) 
 \right)\\
 &=
 \frac{e^{3t}}{\pi}
 \int_0^{\pi} 
 e^{2t\cos \theta}
 \frac{1}{p+q}
  \sum_{j=0}^{p+q-1} 
 \exp \left( 2t\left(
	\cos\left(\frac{p\theta+2\pi j}{p+q} \right) + \cos\left( \frac{q\theta-2\pi j}{p+q}\right)
 \right) \right) 
 \diff \theta,
\end{align*}
where the inner term  
\begin{align*}
\frac{1}{p+q}
  \sum_{j=0}^{p+q-1} 
 \exp \left( 2t\left(
	\cos\left(\frac{p\theta+2\pi j}{p+q} \right) + \cos\left( \frac{q\theta-2\pi j}{p+q}\right)
 \right) \right)
\end{align*}
can be expressed in terms of the functions $\widehat{I}_0,\alpha_{p,q},\beta_{p,q}$ given in \eqref{def:Kpq} and \eqref{def:alphabeta} using the standard trigonometric relations.
\end{proof}
%%%%%%%%%%%%%%%%%%%%%%%%%%%%%%%%%%%%%%%%%%%%%%%%%%%%%%%

For the two classes of zigzag and armchair nanotubes, Theorem \ref{theorem:mgf} directly yields
\begin{corollary}\label{corollary:mfg}
For $t\in\R$ and $p\geq 3$, it holds 
\begin{align*}
m_{\Lambda_{p,p}^*}(t) 
&= 
\frac{e^{3t}}{\pi} \int_0^\pi 
e^{2t\cos\theta}
\widehat{I}_0 \left( 
\cos\left(\frac{\theta}{2}\right), 0, p, p
\right)
\diff \theta
,\\
m_{\Lambda_{p,0}^*}(t) 
&= 
\frac{e^{3t}}{\pi} \int_0^\pi 
e^{2t\cos\theta}
\widehat{I}_0 \left( 
\cos^2\left(\frac{\theta}{2}\right), \frac{\sin\theta}{2}, p, 0
\right)
\diff \theta.
\end{align*}
\end{corollary}
The function $\widehat{I}_{0}$ in Theorem \ref{theorem:mgf} and Corollary \ref{corollary:mfg} can be considered as a discretized version of the integral representation of the modified Bessel function given in \eqref{eq:I0_integral}.
A natural question is the limiting behavior of this discretization as the elements of the chiral vector $(p,q)$, and therefore also the circumference of the nanotube, tend to infinity. 
The answer gives

\begin{proposition}\label{proposition:mfg_limit}
As $p+q\rightarrow \infty$, $\frac{p}{p+q}\rightarrow c\in[0,1]$, we get the following convergence:
\begin{align*}
m_{\Lambda^*_{p,q}}(t) 
\longrightarrow 
\frac{e^{3t}}{\pi} 
\int_0^\pi e^{2t\cos\theta}I_0\left(4t\cos\frac{\theta}{2}\right) \diff \theta.
\end{align*}
\end{proposition}

\begin{proof}
One immediately sees that under the above assumptions
\begin{align*}
\frac{p-q}{p+q}
\xrightarrow[\substack{p+q\rightarrow \infty \\ \frac{p}{p+q}\rightarrow c}]{} 
2c-1.
\end{align*}
Hence, it holds 
\begin{align*}
\alpha_{p,q}(\theta) 
&\xrightarrow[\substack{p+q\rightarrow \infty \\ \frac{p}{p+q}\rightarrow c}]{} 
\cos \left(\frac{\theta}{2} \right)\cos\left( \frac{\theta (2c-1)}{2}\right),\quad \beta_{p,q}(\theta)
\xrightarrow[\substack{p+q\rightarrow \infty \\ \frac{p}{p+q}\rightarrow c}]{} 
\cos\left(\frac{\theta}{2}\right)\sin\left(\frac{\theta(2c-1)}{2}\right),
\end{align*}
which implies 
\begin{align*}
\widehat{I}_0(4t\alpha_{p,q}(\theta),4t\beta_{p,q}(\theta),p,q) 
\xrightarrow[\substack{p+q\rightarrow \infty \\ \frac{p}{p+q}\rightarrow c}]{} 
&\frac{1}{2\pi} \int_0^{2\pi} e^{4t\cos \left(\frac{\theta}{2} \right)\left(\cos\left( \frac{\theta (2c-1)}{2}\right) \cos\varphi + \sin\left(\frac{\theta (2c-1)}{2}\right)\sin\varphi
\right) }
\diff \varphi\\
= &I_0\left(4t\cos\left(\frac{\theta}{2}\right)\right).
\end{align*}
\end{proof}

The combination of the MGF's convergence and the results from \cite{bille2023random} leads to a new integral relation for $I_0$.

\begin{corollary}\label{corollary:Integral formula}
It holds
\begin{align*}
\int_0^1 I_0^3 \left( 2i\sqrt{t\log x}\right) \diff x 
= 
\frac{e^{3t}}{\pi} \int_{0}^{\pi} e^{2t\cos\theta} I_0\left(4t\cos \frac{\theta}{2}\right) \diff \theta, \quad t\in\R.
\end{align*}
\end{corollary}

\begin{proof}
According to \cite[Proposition 5]{bille2023random}, the MGF for the random eigenvalue of the triangular lattice is given by
\begin{align*}
m_{\mathcal{T}}(t)
&=
\varphi_\mathcal{T}(-it) 
=
\int_0^1 I_0^3 \left( 2i\sqrt{t\log x} \right) \diff x,
\end{align*}
where $\varphi_\mathcal{T}(\cdot)$ is the characteristic function of $\mathcal{T}$.
On the other hand, Proposition \ref{proposition:mfg_limit} provides the right-hand side of the claim. 
Since a distribution is uniquely determined by its MGF, and given Proposition \ref{proposition:weak_convergence_Lambdapq}, the claim holds.
\end{proof}

%%%%%%%%%%%%%%%%%%%%%%%%%%%%%%%%%%%%%%%%%%%%%%%%%%%%%%%

%--------------------------------------------------------------------------------------------------------------------------------------------------------------------------------------------------

\section{Probability density function of $\Lambda^*_{p,q}$} \label{sect:density}
Next, we derive explicit formulas for the PDF of dual zigzag ($q=0$) and armchair nanotubes ($p=q$). 
The main idea is to apply the law of total probability across all possible values of $\mathcal{J}_{p,q}$ in both cases. 
For each specific value of $\mathcal{J}_{p,q}$, the distribution of \eqref{theorem1:1} resembles a modified arcsine law supported of intervals determined by the chosen value $\mathcal{J}_{p,q}=j$. 

For the chiral case ($p >q >0$), we introduce a straightforward algorithm for numerically computing the PDF of a dual infinite $(p,q)$--nanotube and illustrate this method on the example $(p,q)=(5,1)$.

%--------------------------------------------------------------------------------------------------------------------------------------------------------------------------------------------------

\subsection{Dual zigzag nanotubes $\boldsymbol{(p\ge 5,~ q=0})$} \label{sect:density_subsect:p0}

\begin{proposition}\label{prop:densityp0}
The PDF of the random eigenvalue $\Lambda_{p,0}^*$ of a dual $(p,0)$--nanotube is given by 
\begin{align}
f_{\Lambda_{p,0}^*}(x) 
&=
\frac{1}{p \pi}
\sum_{j=0}^{p-1} 
\frac{\mathbbm{1}\left\{ x\in \mathcal{I}_j \right\}}{\sqrt{8\left(1+\cos\left( \frac{2\pi j}{p} \right)\right) - 
\left( 3+2\cos\left( \frac{2\pi j}{p} \right) - x \right)^2}} \label{formula:densityp0}
\end{align}
with the intervals 
\begin{align}
\mathcal{I}_j
=
\left( \left( 2\left|  \cos\left( \frac{\pi j}{p} \right)  \right|  - 1\right)^2; \left( 2\left|  \cos\left( \frac{\pi j}{p} \right)  \right|  + 1\right)^2 \right), \quad j = 0,\ldots,p-1. \label{eq:intervals}
\end{align}
\end{proposition}

\begin{proof}
By Theorem \ref{thm:RV_pq}, we have the following identities:
\begin{align*}
\Lambda_{p,0}^*
\eqd 
\left| 1 + e^{i\frac{2\pi \mathcal{J}_{p,0}}{p}} + e^{i\mathcal{U}} \right|^{2}
\eqd
\left| r_{p, \mathcal{J}_{p,0}} + e^{i \mathcal{U}} \right|^{2}
=
1  + 2r_{p,\mathcal{J}_{p,0}}\cos \mathcal{U} + r_{p,\mathcal{J}_{p,0}}^2
\end{align*} 
with independent random variables $\mathcal{U}\sim U (0,\pi)$, $\mathcal{J}_{p,0}\sim U\lbrace 0,\ldots, p-1 \rbrace$, and 
\begin{align*}
r_{p,\mathcal{J}_{p,0}}
:= 
\left| 1 + e^{i\frac{2\pi \mathcal{J}_{p,0}}{p}}  \right| 
= 
\sqrt{ 2 + 2\cos\left( \frac{2\pi \mathcal{J}_{p,0}}{p} \right)}
=
2\left|  \cos\left( \frac{\pi \mathcal{J}_{p,0}}{p} \right)  \right|.
\end{align*}
Using this and the law of total  probability, we obtain
\begin{align*}
f_{\Lambda^*_{p,0}}(x)
&= 
\frac{1}{p}\sum_{j=0}^{p-1} f_{\left.\Lambda_{p,0}^*\right| \mathcal{J}_{p,0}=j}(x).
\end{align*}
Furthermore, using the transformation formula, the conditional PDF we need is given by the arcsine law:
\begin{align*}
f_{\left.\Lambda_{p,0}^*\right| \mathcal{J}_{p,0} = j}(x) 
= 
f_{1 + r_{p,j}^2 + 2r_{p,j}\cos \mathcal{U}} (x)
&=
\frac{1}{\pi\sqrt{\left( x - \left( r_{p,j} - 1 \right)^2 \right)\left(\left( r_{p,j} + 1 \right)^2 - x  \right)}}
\mathbbm{1}\lbrace x\in \left( (r_{p,j} - 1)^2; (r_{p,j} + 1)^2 \right) \rbrace\\
&= 
\frac{1}{\pi \sqrt{4r_{p,j}^2 - \left( r_{p,j}^2 + 1 - x \right)^2}} 
\mathbbm{1}\left\{ x\in \left( \left(r_{p,j} - 1\right)^2; \left(r_{p,j} + 1\right)^2 \right) \right\}.
\end{align*} 
\end{proof}

\begin{example} \label{ex:50}
In the special case where $p=5$, $q=0$ the numerical intervals \eqref{eq:intervals} are 
\begin{align*}
\mathcal{I}_0 &= \left(0.382; 6.854 \right),  \quad
\mathcal{I}_1 = \left( 0.146; 2.618 \right), \quad
\mathcal{I}_2 = \left( 0.146; 2.618 \right),  \quad
\mathcal{I}_3 = \left( 0.382; 6.854 \right), \quad
\mathcal{I}_4 = \left( 1; 9  \right).
\end{align*}
These intervals match our numerical results illustrated in Figure \ref{fig:nanotube_density_fit}.

Additionally, it should be noted that the PDF has a logarithmic pole at each endpoint of the intervals $\mathcal{I}_j$ for $j = 0, \ldots, 4$. 
The heights of the blue peaks shown in Figure \ref{fig:nanotube_density_fit} are due to computational inaccuracies.

\begin{figure}[H]
\centering
\hspace*{-1cm}\includegraphics[scale=0.55]{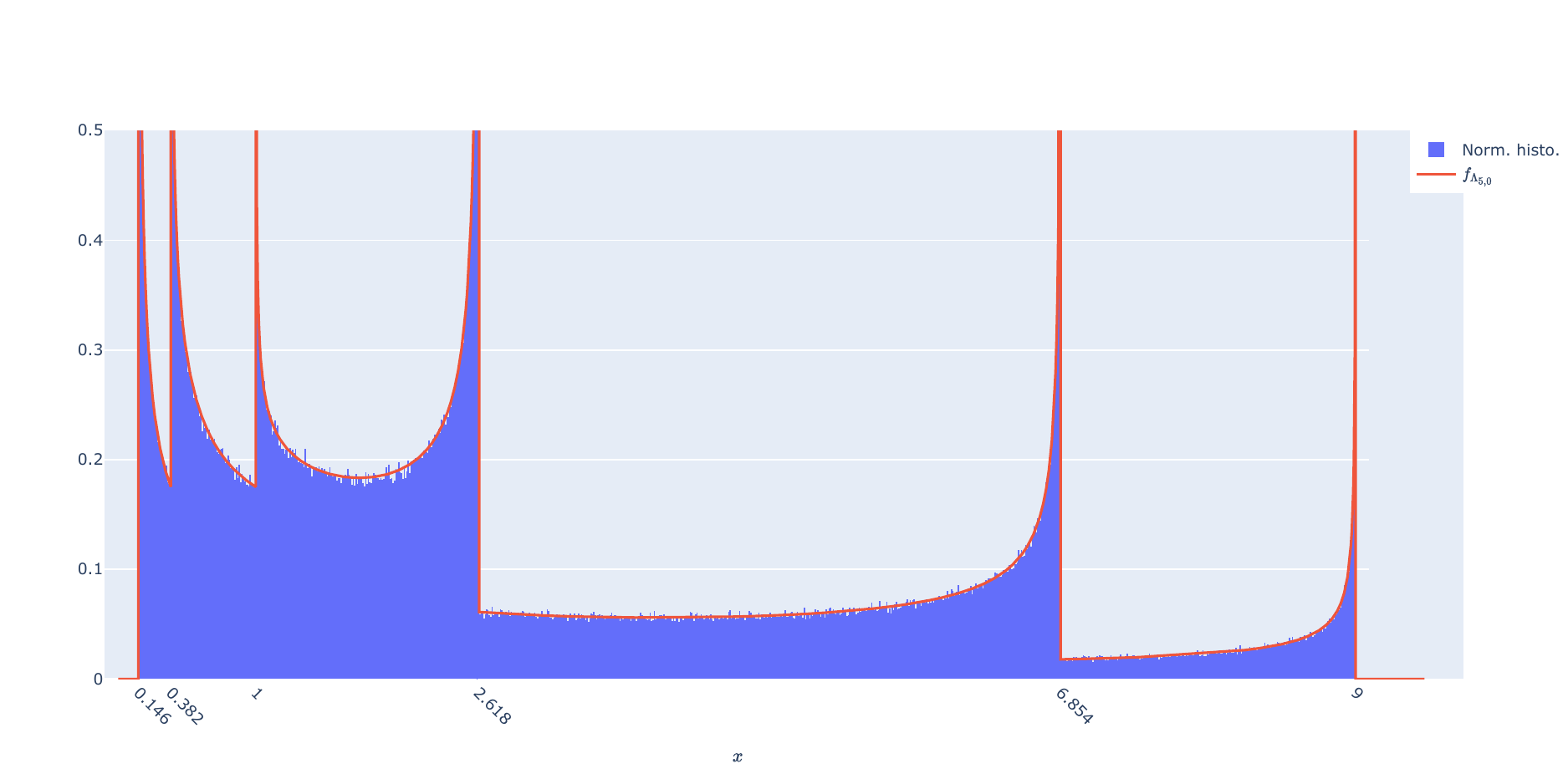}
\caption{Histogram of eigenvalues of the dual graph of facets of  $(5,0)$--nanotube $C_{100000}$ (blue) and the graph of $f_{\Lambda_{5,0}^*}$ (red).}\label{fig:nanotube_density_fit}
\end{figure}
\end{example}

%--------------------------------------------------------------------------------------------------------------------------------------------------------------------------------------------------

\subsection{Dual armchair nanotubes $\boldsymbol{(p=q})$} \label{sect:density_subsect:pp}
\begin{proposition}
For $p\geq 2$, the PDF of the random eigenvalue $\Lambda_{p,p}^*$ of a dual $(p,p)$--nanotube is given by 
\begin{align}
f_{\Lambda_{(p,p)}^*}(x)
&= 
\frac{1}{p}\left( \frac{1}{2}f_0(x) +  \sum_{j=1}^{p-1} 
f_{j}(x) 
\right), 
\label{eq:(p,p)_sumformula}
\end{align}
where
\begin{align}
&f_{0}(x) 
= 
\frac{\mathbbm{1}\left\{0 < x < 9 \right\}}{\pi\sqrt{x\left(3+2\sqrt{x}-x\right)}}
+
\frac{\mathbbm{1}\lbrace 0 < x < 1 \rbrace}{\pi\sqrt{x\left(3 - 2\sqrt{x}-x\right)}}, \label{formula:f_1}\\
&f_j(x) = f_{4\cos^2\left(\mathcal{V}\right) + 4 a_j \cos \left( \mathcal{V}\right) + 1}(x)
= 
\begin{cases}
	%1.Case
	\frac{\mathbbm{1}\left\{ 1 \leq x < 5 + 4 a_j\right\}
	}{\pi
	\sqrt{\left( 4 - \left(-a_j + \sqrt{a_j^2 + x - 1} \right)^2 \right)
	\left( a_j^2 + x - 1 \right)}	
	}, &2j\leq p,\\
	%2.Case
	\frac{1}{\pi\sqrt{a_j^2 + x - 1 }}
	\left(
	\frac{\mathbbm{1}\left\{ 1 - a_j^2 < x < 1 \right\}}
	{\sqrt{ 4 - \left( a_j + \sqrt{a_j^2 + x - 1} \right)^2 }} 
	+
	\frac{\mathbbm{1}\left\{ 1 - a_j^2 < x < 5 + 4a_j \right\}
	}{\sqrt{ 4 - \left( - a_j + \sqrt{a_j^2 + x - 1} \right)^2 }}
	\right)	
	 , &2j >p,
\end{cases}\label{formula:f_j}
\end{align}
with $a_j:= \cos\left(\sfrac{j\pi}{p}\right)$, $j\in\lbrace 1,\ldots,p-1\rbrace$.
\end{proposition}

\begin{proof}
First, note that formula (\ref{eq:(p,p)_sumformula}) is a direct consequence of the law of total probability, where we sum up over all possible values $j$ of $\mathcal{J}_{p,p}\sim U\lbrace 0,1,\ldots, 2p-1\rbrace$. \\
Let $\mathcal{U}\sim U(-\pi,\pi),~  \mathcal{V}\sim U\left(0,\sfrac{\pi}{2}\right)$ be two independent random variables.
Due to Theorem \ref{thm:RV_pq}, it holds
\begin{align}
\Lambda_{(p,p)}^* 
&\overset{d}{=}
3+2\left( 
\cos \mathcal{U} + \cos\left( \frac{p\, \mathcal{U} + 2\pi \mathcal{J}_{p,p}}{2p} \right)  + 
\cos\left( \frac{p\, \mathcal{U} - 2\pi \mathcal{J}_{p,p}}{2p} \right) \right) \notag\\
&= 
4\cos^2\left(\frac{\mathcal{U}}{2}\right) + 4\cos\left(\frac{\pi \mathcal{J}_{p,p}}{p}\right) \cos \left(\frac{\mathcal{U}}{2}\right) + 1 
\overset{d}{=} 
4\cos^2 \mathcal{V} + 4\cos\left(\frac{\pi \mathcal{J}_{p,p}}{p}\right) \cos \mathcal{V} + 1, \label{formula:polynomial}
\end{align}
resulting in a polynomial of degree 2 in $\cos \mathcal{V}$, 
where for given $p$ the coefficient $4\cos\left(\sfrac{\pi \mathcal{J}_{p,p}}{p}\right)$ takes only $p+1$ different values although $\mathcal{J}_{p,p}$ takes $2p$ different values, which is a direct consequence of the periodicity of cosine.
This means, we need to consider values $j\in\left\{ 0,\ldots,p \right\}$ only.

For $j \in\lbrace 0,p\rbrace$, simple computations yield
\begin{align*}
f_{4\cos^2 \mathcal{V} + 4\cos\left(\frac{\pi j}{p}\right) \cos \mathcal{V} + 1}(x) =
\begin{cases}
f_{\left(1 + 2\cos \mathcal{V} \right)^2}(x) 
= 
\frac{\mathbbm{1}\left\{1\leq x < 9 \right\}}{\pi\sqrt{x\left(3+2\sqrt{x}-x\right)}}, &j = 0,\\
f_{\left(1 - 2\cos \mathcal{V}  \right)^2}(x) 
=
\frac{\mathbbm{1}\lbrace 0 < x \leq 1 \rbrace}{\pi\sqrt{x\left(3+2\sqrt{x}-x\right)}}+\frac{\mathbbm{1}\lbrace 0 < x < 1 \rbrace}{\pi\sqrt{x\left(3 - 2\sqrt{x}-x\right)}},\quad &j = p,
\end{cases}
\end{align*}
which directly yields (\ref{formula:f_1}).

If $2j=p$, we get
\begin{align*}
f_{4\cos^2 \mathcal{V} + 4\cos\left(\frac{\pi j}{p}\right) \cos \mathcal{V} + 1}(x) 
= 
f_{4\cos^2 \mathcal{V} + 1}(x) 
= 
\frac{1}{\pi\sqrt{(5-x)(x-1)}}\mathbbm{1}\left\{ 1 < x < 5\right\},
\end{align*}
which is the arcsine distribution on the interval $(1,5)$. 
Note that this case occurs only if $p$ is an even number. 

For the next two cases, note that the function $\cos: (0,\sfrac{\pi}{2}) \to (0,1)$ is monotonically decreasing. 
Hence, replacing $x$ by the random variable $\mathcal{V}$ yields the standard cosine distribution with PDF:
\begin{align*}
f_{\cos \mathcal{V}}(x) 
= 
\frac{2}{\pi\sqrt{1-x^2}} \mathbbm{1}\lbrace 0 \leq x < 1 \rbrace.
\end{align*}

\noindent If $2j < p$, which is equivalent to $\cos\left(\sfrac{j\pi}{p}\right) > 0$, it becomes evident that the functions
\begin{align*}
h_j(x)&:= 4x^2 + 4\cos\left(\frac{j\pi}{p}\right)x + 1,\quad j\in\lbrace 1,\ldots,p-1 \rbrace
\end{align*}
are monotonically increasing on $[0,1]$, and, therefore, bijective with inverse 
\begin{align*}
h_j^{-1}(y) &= \frac{1}{2}\left(-\cos\left(\frac{ j\pi }{p}\right) +\sqrt{\cos^2\left(\frac{ j\pi }{p}\right)+y-1} \right).
\end{align*}
Their derivatives $\left(h_j^{-1}\right)'$ have the property
\begin{align*}
\left(h_j^{-1}\right)'(y) &= \frac{1}{4\sqrt{\cos^2\left(\frac{j\pi}{p}\right)+y-1}} > 0, \quad y\in\left[1,5+4\cos\left( \frac{j\pi}{p} \right)\right]. 
\end{align*}
Then, the pdf of $h_j\left(\cos V \right)$ with $2j < p$ can be easily computed applying the transformation formula qs
\begin{align}
f_{h_j\left(\cos \mathcal{V}\right)}(x) 
&= 
f_{4\cos^2 \mathcal{V} + 4\cos\left(\frac{\pi j}{p}\right) \cos \mathcal{V} + 1}(x)
= 
f_{\cos \mathcal{V}}\left(h_j^{-1}(x)\right) \cdot \left| \left(h_j^{-1}\right)'(x) \right|, \label{density_h_j}
\end{align}
yielding the first case in (\ref{formula:f_j}).

If $2j>p$, i.e., it holds that $\cos\left(\sfrac{j\pi}{p}\right)<0$, the function $h_j$ is not bijective, and, therefore, the procedure presented in the previous case can not be applied directly.  
However, we bypass this problem by splitting the function $h_j$ into two bijective subfunctions as illustrated in Figure \ref{fig:h_j_decomposition}:
\begin{align*}
h_j(x) = h_j^-(x) + h_j^+(x) ,
\end{align*}
where 
\begin{align*}
 h_j^-(x) = h_j(x) \mathbbm{1}\lbrace h_j'(x)<0 \rbrace, \quad h_j^+(x) = h_j(x) \mathbbm{1}\lbrace h_j'(x)>0 \rbrace.
\end{align*}
Then, the PDF of $h_j(\cos \mathcal{V})$ is just the sum of the densities of $h_j^-(\cos \mathcal{V})$ and $h_j^+(\cos \mathcal{V})$, which yields in the second case of (\ref{formula:f_j}).

\begin{figure}
\begin{center}
\includegraphics[scale=0.4]{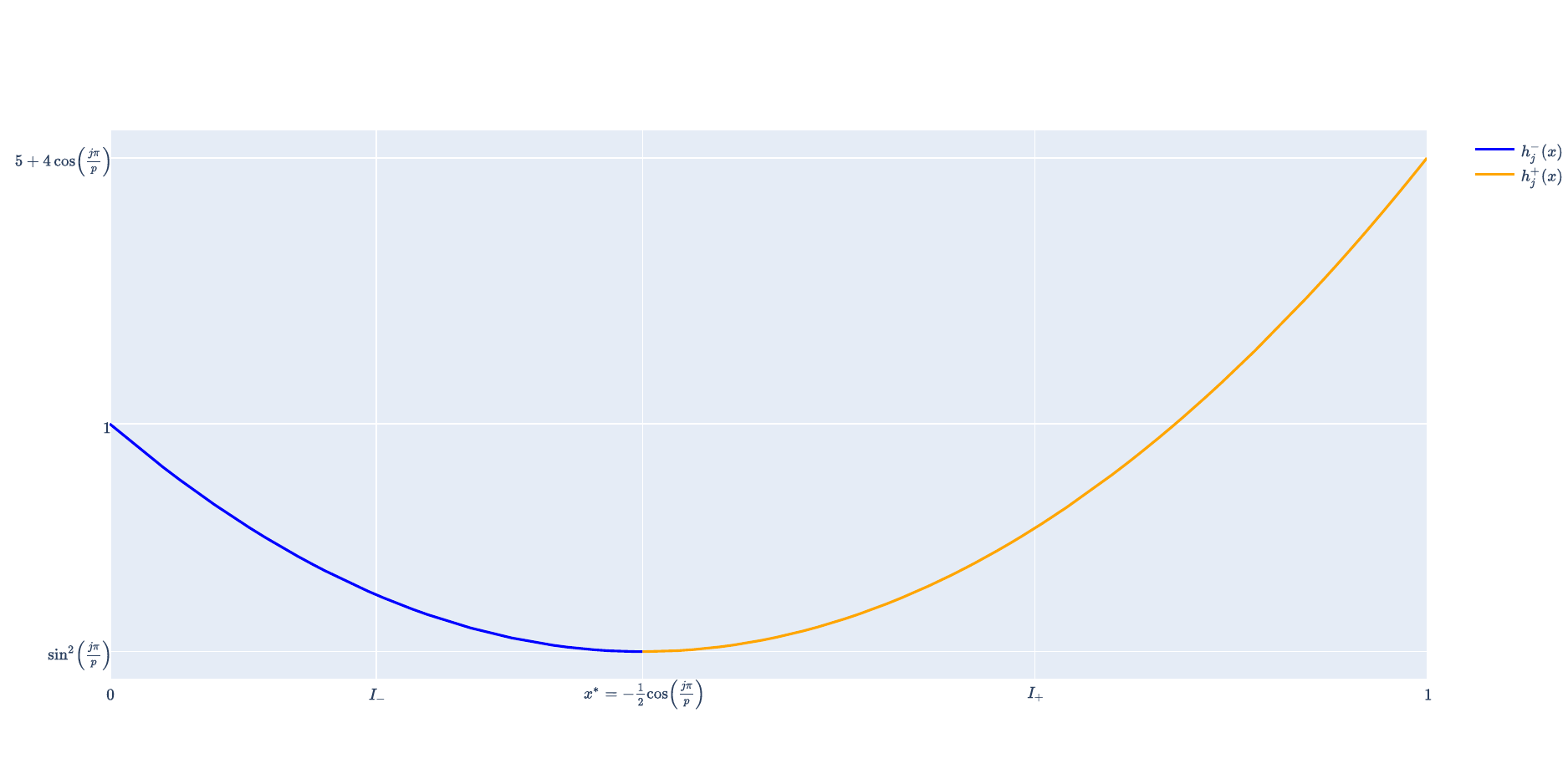}
\caption{Decomposition of $h_j$ into two bijective functions.}\label{fig:h_j_decomposition}
\end{center}
\end{figure}

\end{proof}

\begin{example}
For every $n=60+10r$ with $r\in\N_{0}$, there exists a finite $(5,5)$--nanotube among the isomers of $C_n$, cf. Construction \ref{construction:55--nanotube} in Appendix.
The smallest example $n = 60$ with no hexagonal belts and two caps has the second largest normalized Newton polynomials of order $k$ for every $k\geq 4$.
 
As $r\rightarrow\infty$, the empirical PDF of a random eigenvalue of these nanotubes converge to $\Lambda_{5,5}^*$ with explicit PDF given as
\begin{align*}
f_{\Lambda_{(5,5)}^*}(x)&= \frac{1}{10}\sum_{j=0}^{5} f_j(x), 
%\label{eq:(5,5)_sumformula}
\end{align*}
where
\begin{align*}
f_{0}(x)
&= 
\frac{\mathbbm{1}\lbrace 0 < x < 9 \rbrace}{\pi\sqrt{x\left(3+2\sqrt{x}-x\right)}},\quad\quad  
f_5(x)
= 
\frac{\mathbbm{1}\lbrace 0 < x < 1 \rbrace}{\pi\sqrt{x\left( 3-2\sqrt{x}-x \right)}},\\
f_1(x)
&=
\frac{4\cdot \mathbbm{1}\left\{ \frac{5-\sqrt{5}}{8} < x < 6+\sqrt{5} \right\}}{\pi\sqrt{\left(1-\frac{1}{64}\left(1+\sqrt{5}-\sqrt{16x+2\sqrt{5}-10}\right)^2\right)\left(16x+2\sqrt{5}-10\right)}},\\
f_2(x)
&= 
\frac{4\cdot \mathbbm{1}\left\{ \frac{5+\sqrt{5}}{8} < x < 4+\sqrt{5} \right\}}{\pi\sqrt{\left(1-\frac{1}{64}\left(1-\sqrt{5}+\sqrt{16x-2\sqrt{5}-10}\right)^2\right)\left(16x-2\sqrt{5}-10\right)}},\\
f_3(x)
&= 
\frac{4\cdot \mathbbm{1}\left\{ \frac{5+\sqrt{5}}{8} < x < 6-\sqrt{5} \right\}}{\pi\sqrt{\left(1-\frac{1}{64}\left(1-\sqrt{5}-\sqrt{16x-2\sqrt{5}-10}\right)^2\right)\left(16x-2\sqrt{5}-10\right)}} ,\\
f_4(x)
&= 
\frac{4\cdot \mathbbm{1}\left\{ \frac{5-\sqrt{5}}{8} < x < 4-\sqrt{5} \right\}}{\pi\sqrt{\left(1-\frac{1}{64}\left(1+\sqrt{5}+\sqrt{16x+2\sqrt{5}-10}\right)^2\right)\left(16x+2\sqrt{5}-10\right)}}.
\end{align*}
Its graph is given in Figure \ref{fig:pdf_lambda_55}.
\end{example}

\begin{figure}[H]
\begin{center}
\includegraphics[scale=0.4]{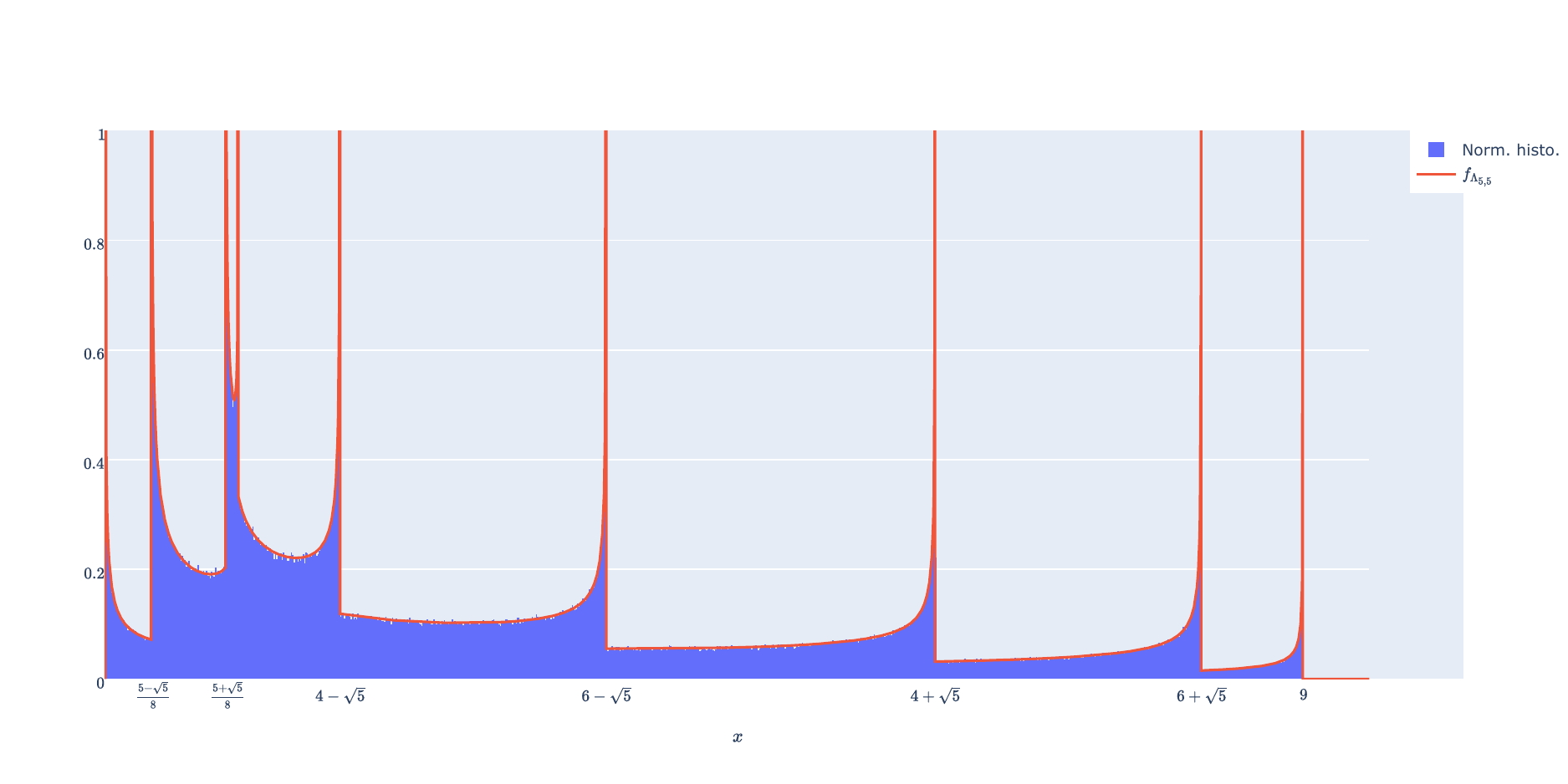}
\caption{Normalized histogram (blue shaded area) depicting $10^7$ samples of $\Lambda_{(5,5)}^*$, alongside the graph of function $f_{\Lambda_{(5,5)}^*}$ (red line). Note that the first and fourth peak counted by the left are observed for $x = 0$ and $x = 1$.}\label{fig:pdf_lambda_55}
\end{center}
\end{figure}

%--------------------------------------------------------------------------------------------------------------------------------------------------------------------------------------------------

\subsection{Chiral dual nanotubes $\boldsymbol{(p>q>0})$} \label{sect:density_subsect:pq}

In contrast to the previous cases of zigzag and armchair nanotubes, deriving a general closed-form formula for the PDF of $\Lambda_{p,q}^*$ when $p>q>0$ is not feasible.
However, by conditioning on the value of $\mathcal{J}_{p,q}$ using the law of total probability, and expressing $\Lambda_{p,q}^*$ as a function of a truncated arcsine distribution, we can construct an algorithm to numerically compute the desired PDF. 
Let $\mathcal{J}_{p,q}\sim U\lbrace 0,\ldots,p+q-1\rbrace$ and 
$\mathcal{U}\sim U(0,\pi)$ be independent. 
Then the algorithm proceeds as follows:\\
For every fixed $\mathcal{J}_{p,q} = j\in\lbrace 0,\ldots,p+q-1\rbrace$:
\begin{enumerate}[(i)]
\item Apply basic trigonometric calculus and rewrite the right--hand side of \eqref{theorem1:1} as a function $\varphi_j$ depending solely on 
$\mathcal{V}_{p,q} := \cos \frac{\mathcal{U}}{p+q}$.\\
Since $p+q>0$, the function $\cos\frac{x}{p+q}$ is bijective for $x\in (0,\pi)$. Therefore, it holds:
\begin{equation*}
f_{\mathcal{V}_{p,q}}(x)
%=
%f_{\cos\frac{\mathcal{U}}{p+q}}(x) 
=
\frac{p+q}{\pi\sqrt{1-x^2}}\mathbbm{1}_{\left(\cos\frac{\pi}{p+q},1\right)}(x) .
\end{equation*}
\item Numerically compute any local extrema of $\varphi_j$ within the interval 
$\left(\cos \frac{\pi}{p+q},1\right)$.
\item Divide the interval $\left(\cos\frac{\pi}{p+q},1\right)$ into subintervals based on the extrema found in the previous step, ensuring that $\varphi_j$ is bijective on each subinterval. 
\item  Numerically compute the PDF $\widehat{f}_{\varphi_j(\mathcal{V}_{p,q})}$ of $\varphi_j\left(\mathcal{V}_{p,q}\right)$ piecewise on these subintervals.
\end{enumerate}
Finally, aggregating these transformations yields:
\begin{align*}
f_{\Lambda_{p,q}^*}(x) 
\approx
 \frac{1}{p+q}\sum_{j=0}^{p+q-1} \widehat{f}_{\varphi_j(\mathcal{V}_{p,q})}(x).
\end{align*}

We now apply the above procedure to the infinite dual $(5,1)$--nanotube. 
Firstly, notice that the second isomer of $C_{60}$ (according to the spiral enumeration) is a finite $(5,1)$--nanotube, which is the smallest possible finite nanotube aside from the $(5,0)$--nanotube. 
Secondly, according to \cite{Grimme17}, the $(5,1)$--nanotube is among the three $C_{60}$ isomers with the highest relative molecular energy, making it one of the least chemically stable isomers. 
All subsequent computations can be verified and adopted to other chiral dual nanotubes using our Python notebook \cite{PythonBille24}.

\begin{example}[PDF of $\Lambda_{5,1}^*$]
We aim to derive the PDF of
\begin{align*}
\Lambda_{5,1}^* &\overset{d}{=} 
3 + 2\left( \cos \mathcal{U} + \cos\left(\frac{5U+2\pi \mathcal{J}_{5,1}}{6}\right) + \cos\left(\frac{\mathcal{U}-2\pi \mathcal{J}_{5,1}}{6}\right) \right)
\end{align*}
with independent $\mathcal{U}\sim U(0,\pi)$ and $\mathcal{J}_{5,1}\sim U\lbrace 0,1,\ldots,5 \rbrace$.
Applying the relation for $\cos(x \pm y)$ and the multiple-angle formulas
\begin{align*}
\cos (nx) 
&=
2^{n-1}\cos^n(x) + n\sum_{k=1}^{\lfloor \sfrac{n}{2} \rfloor}
\frac{(-1)^k}{k} \binom{n-k-1}{k-1} 2^{n-2k-1} \cos^{n-2k}(x),\\
\sin(nx) 
&=
\sin(x)P_{n-1}(\cos(x))
\end{align*}
with $n\in\N$ and $x\in\R$, where $P_{n}$ is the $n$-th Chebyshev polynomial of the second kind, which can be defined (similar to $T_n$) as the unique solution of
\begin{align*}
P_n(\cos \theta) \sin \theta 
= 
\sin((n+1)\theta),
\end{align*}
yields the expression $\Lambda_{5,1}^* \eqd \varphi_{\mathcal{J}_{5,1}}\left(\mathcal{V}_{5,1}\right)$, where
\begin{align}
\varphi_{\mathcal{J}_{5,1}}\left(\mathcal{V}_{5,1}\right)
:&=
64 \mathcal{V}_{5,1}^6 + 
32 c_{\mathcal{J}_{5,1}} \mathcal{V}_{5,1}^5 - 
32\left(3 + d_{\mathcal{J}_{5,1}} \sqrt{1-\mathcal{V}_{5,1}^2}\right)\mathcal{V}_{5,1}^4 - 
40 c_{\mathcal{J}_{5,1}} \mathcal{V}_{5,1}^3 \notag \\ 
&+ 
12\left( 2 d_{\mathcal{J}_{5,1}} \sqrt{1-\mathcal{V}_{5,1}^2} + 
3 \right)\mathcal{V}_{5,1}^2 + 
12 c_{\mathcal{J}_{5,1}} \mathcal{V}_{5,1} + 
1, \label{formula:V_polynom}
\end{align}
the PDF of $\mathcal{V}_{5,1}:=\cos \sfrac{\mathcal{U}}{6}$ is given as
\begin{align*}
f_{\mathcal{V}_{5,1}}(x)
=
%f_{\cos\frac{\mathcal{U}}{6}}(x)
%=
\frac{6}{\pi\sqrt{1-x^2}} \mathbbm{1}_{\left(\sfrac{\sqrt{3}}{2},1\right]}(x),
\end{align*}
and
\begin{align*}
c_{\mathcal{J}_{5,1}} := \cos\left(\frac{2\pi \mathcal{J}_{5,1}}{6}\right), \qquad
d_{\mathcal{J}_{5,1}} := \sin\left(\frac{2\pi \mathcal{J}_{5,1}}{6}\right).
\end{align*}
The final step is to compute the PDF of $\varphi_{\mathcal{J}_{5,1}}(\mathcal{V}_{5,1})$ using the law of total probability
\begin{align*}
f_{\varphi_{\mathcal{J}_{5,1}}(\mathcal{V}_{5,1})}(x) 
= 
\frac{1}{6} \sum_{j=0}^{5} f_{\varphi_{\mathcal{J}_{5,1}}(\mathcal{V}_{5,1})|\mathcal{J}_{5,1}=j}(x),
\end{align*}
and the transformation formula 
\begin{align*}
f_{\varphi_{\mathcal{J}_{5,1}}(\mathcal{V}_{5,1})|\mathcal{J}_{5,1}=j}(x) 
= 
f_{\mathcal{V}_{5,1}}\left(\varphi_j^{-1}(x)\right) \cdot \left| \frac{\diff}{\diff x} \varphi_j^{-1}(x) \right|.
\end{align*}
Here, for every fixed value $j\in\lbrace 0,\ldots,5\rbrace$, the function $\varphi_j$ has at most one local extremum, see Figure \ref{fig:varphi_j}.
Table \ref{tab:example} gives an overview of all positions and values of extrema of $\varphi_j$ on $\left[\sfrac{\sqrt{3}}{2},1\right)$. 
It is not hard to see that $\varphi_0$ is monotonically increasing, and, therefore, has no local extrema. 
Therefore, cells in Table \ref{tab:example} for $j=0$ are left blank.
It is noteworthy that we have two pairs $(j_1,j_2)\in\lbrace (1,5),~ (2,4) \rbrace$, such that $j_1+j_2 = 6$ with $\varphi_{j_1}(\sfrac{\sqrt{3}}{2}) = \varphi_{j_2}(\sfrac{\sqrt{3}}{2})$ and $\varphi_{j_1}(1) = \varphi_{j_2}(1)$.
$\varphi_{j_1}$ has a local minimum whereas $\varphi_{j_2}$ has a local maximum in both cases.

\begin{figure}[H]
\begin{center}
\includegraphics[scale=0.5]{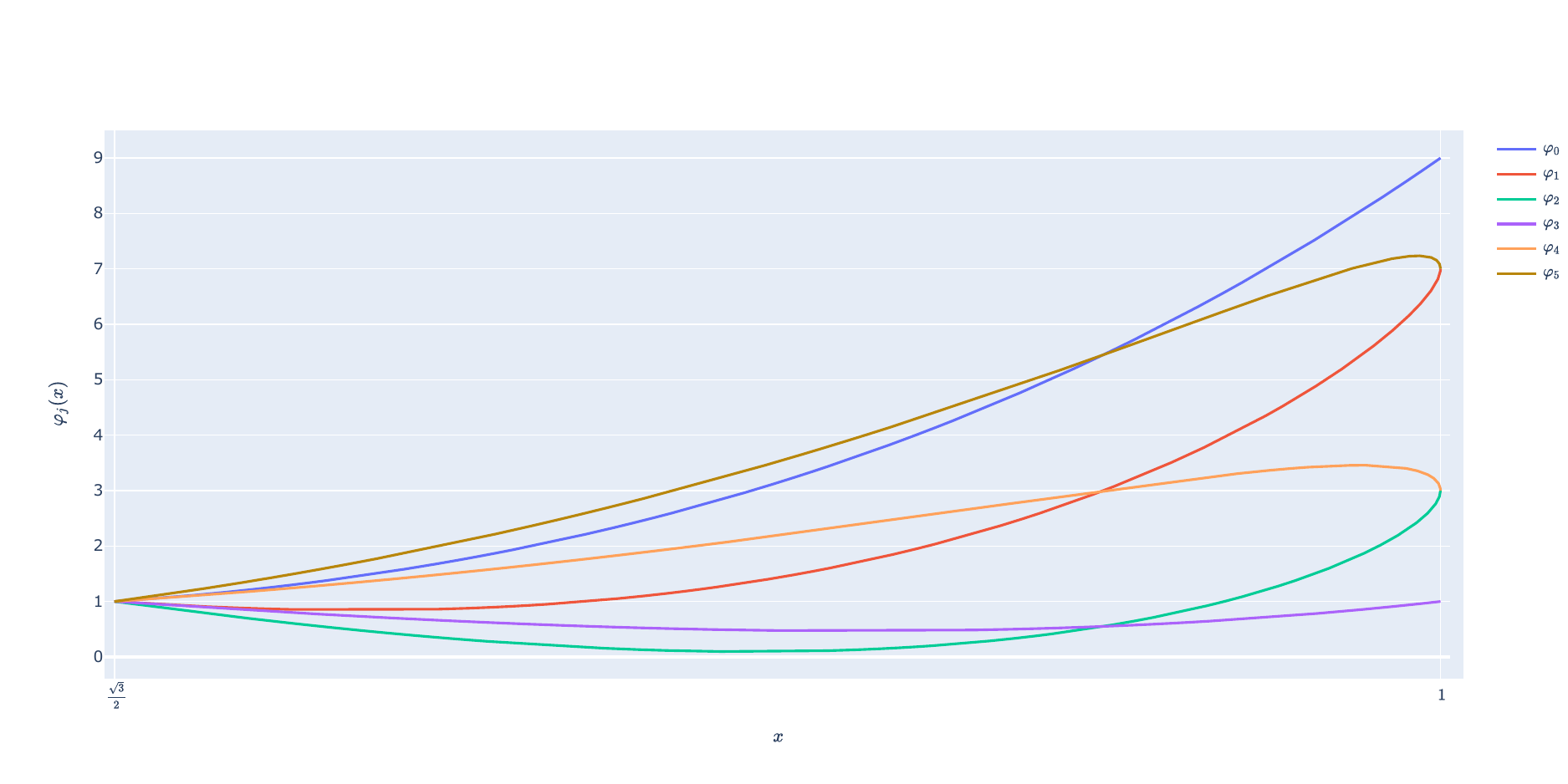}
\caption{The family of functions $\lbrace \varphi_j \rbrace_{j\in\lbrace 0,\ldots,5\rbrace}$ on the interval $\left[\sfrac{\sqrt{3}}{2},1\right)$.}\label{fig:varphi_j}
\end{center}
\end{figure}

Another observation is concerning the shape of the numerical PDF of $\varphi_j(\mathcal{V}_{5,1})$.
Although the distribution described in \eqref{formula:V_polynom} appears to be a tedious polynomial-like function, our numerical results suggest that these distributions follow a much simpler distribution as one single or a sum of two truncated arcsine distributions. 

Finally, we can invert $\varphi_j$ piece-wise, compute the transformed probability density function and sum them up on the corresponding intervals.

Using Python software \cite{PythonBille24}, the numerical computation of $\varphi_j^{-1}$ and $\frac{\diff}{\diff x}\varphi_j(x)$ leads to the function illustrated as the red line in Figure \ref{fig:N_51_pdf}.
The blue bars represent the normalized histogram of $N=10^{6}$ i.i.d. samples of the random eigenvalue $\Lambda_{5,1}^*$, simulated straightforwardly according to the representation given in Theorem \ref{thm:RV_pq}.

\begin{table}[h]
\centering
\begin{tabular}{c||c|c|c|c|c|c}
\hline
$j$ & $0$ & $1$ & $2$ & $3$ & $4$ & $5$ \\ \hline
$x_*$ & -- & 0.890885 & 0.930533 & 0.941337 & 0.991806 & 0.997728 \\ \hline
$\varphi_j\left(x_*\right)$ & -- & 0.843372 & 0.094556 & 0.467574 & 3.45796 & 7.23622 \\ \hline
\end{tabular}
\caption{Positions and values of extrema of $\varphi_j$ on $\left[\sfrac{\sqrt{3}}{2} ,1\right)$.}
\label{tab:example}
\end{table}

\begin{figure}[H]
\begin{center}
\includegraphics[scale=0.5]{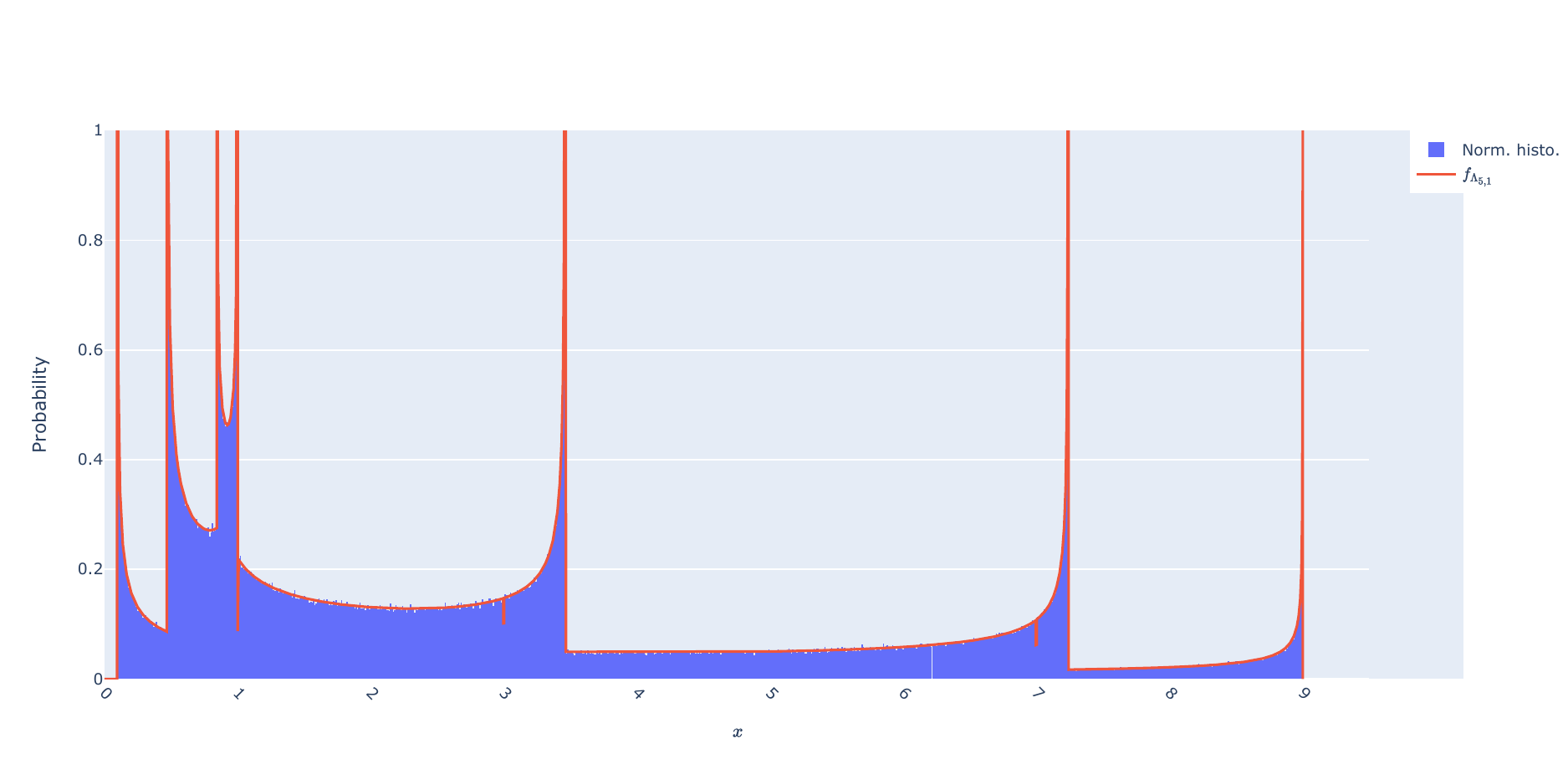}
\caption{Normalized histogram of $10^{6}$ random eigenvalues $\Lambda_{5,1}^*$ (simulated by \eqref{theorem1:1}) (blue) and the numerical calculation of density $ f_{\Lambda_{5,1}^*}$ (red).}\label{fig:N_51_pdf}
\end{center} 
\end{figure}

Note that in Figure \ref{fig:N_51_pdf}, the drops in the red curve at $x\in\lbrace 1,3,7 \rbrace$ are attributable to numerical inaccuracy.  

\end{example}

\section{Summary} \label{sect:summary}

This paper explores the spectrum of dual infinite $(p,q)$--nanotubes by investigating the distribution of their random eigenvalues $\Lambda_{p,q}^*$. 
The study derives a general representation for these random eigenvalues and provides specific formulas for zigzag ($q = 0$) and armchair ($p = q$) nanotubes.

For general dual infinite $(p,q)$--nanotubes, the paper employs a formula from \cite{Cotfas00} for the number of closed paths with $k$ steps. 
Based on that, the random eigenvalues $\Lambda_{p,q}^*$ are shown to have a specific distribution involving independent uniformly distributed random variables $\mathcal{U}$ and $\mathcal{J}_{p,q}$. 
This distribution simplifies to more practical formulas in zigzag and armchair cases.

The paper also analyzes the asymptotic behavior of $\Lambda_{p,q}^*$ as $p+q\to\infty$ such that $\frac{p}{p+q}\to c\in[0,1]$, showing convergence to the random eigenvalues of the triangular lattice. 
Additionally, it provides three distinct representations for the moments of $\Lambda_{p,q}^*$. 

The MGF of $\Lambda_{p,q}^*$ can be expressed as an integral involving a discrete form of the modified Bessel function, aligning with previous results for the hexagonal and triangular lattices.
As a direct consequence, this yields a novel identity of the integral over the modified Bessel function.

Finally, explicit formulas for the PDFs of $\Lambda_{p,q}^*$ are given for zigzag and armchair nanotubes. 
For chiral nanotubes, the paper proposes a numerical algorithm to compute their PDFs and demonstrates its application on the dual infinite $(5, 1)$--nanotube.

\subsection*{Acknowledgement}

The authors express their gratitude to Vsevolod L. Chernyshev and Satoshi Kuriki for valuable discussions and helpful suggestions regarding appropriate references.

\newpage

\bibliography{literature}{}
\bibliographystyle{abbrv}

\newpage 
\appendix
\section*{Appendix}

\begin{proposition}[Proof of  representation \eqref{eq:moments2}]
For $(p,q)\in\mathbf{N}$ and $k\in\N$, it holds
\begin{align*}
\mu_k\left(\Lambda_{p,q}^*\right)
&=
\sum_{k_1+k_2+k_3=k} 
\binom{k}{k_1,k_2,k_3}^2 
\left( 1 + 
2\sum_{j=1}^{\left\lfloor \frac{k_1}{p+q} \right\rfloor} 
\frac{\binom{2k_1-jq}{k_1+jp} \binom{k}{k_1-jq}}{\binom{k}{k_1} \binom{2(k-k_1)}{k-k_1}} 
\right).
\end{align*}
\end{proposition}

\begin{proof}
Di Crescenzo et al. presented in \cite[Proposition 3.2]{DiCre} formulas 
for the number of all possible walks between two given vertices of a hexagonal (triangular) lattice. 
Since we study closed paths we choose the starting vertex as the origin, and in a $(p,q)$--nanotube this starting vertex coincides with every vertex with the coordinates $j(p,q)$ for all $j\in\Z$. 
For the same reason, one sums the triple integrals over all integer values of $j$  in \eqref{eq:fCotfas}. 
Since any vertex $v$ of a $(p,q)$--nanotube is identified with $v+j\sqrt{3}(p,q)^\intercal$ for every $j\in\Z$, we count the number of closed walks on $\mathcal{H}$ from $v$ to $v+j\sqrt{3}(p,q)^\intercal$ for all $j\in\Z$.
Due to the symmetry properties of the hexagonal lattice shown in \cite[Corollary 1]{DiCre}, we need to consider $0\leq q \leq p$ and $j\geq 0$ only.  
Using \cite[Proposition 2]{DiCre} one gets

\begin{align*}
\mu_k\left(\Lambda_{p,q}^*\right)
&=
\sum_{j\in\Z} \mathcal{P}_{jp,jq}^{2k}(\mathcal{H})
=
\mu_k\left(T\right)
+
2\sum_{j=1}^\infty 
\sum_{z=jq}^{k-jp}
\binom{k}{z}\binom{k}{jp+z}
\binom{z}{jq}{}_2F_1
\left(
\begin{matrix}
-jp-z,jq-z\\
1+jq
\end{matrix}
;1\right) \\
&=
\mu_k\left(T\right) + 
2 \sum_{j=1}^{\left\lfloor \frac{k}{p+q} \right \rfloor} \sum_{z=jq}^{ k-jp} \binom{k}{z}\binom{k}{jp+z}\binom{z}{jq}\frac{(jq)!(2z+jp)!}{z!(z+j(p+q))!},
\end{align*}

where $\mathcal{P}_{x,y}^{2k}(\mathcal{H})$ denotes in \cite{DiCre} the number of walks of length $2k$ from the origin to a vertex $(x,y)$ on $\mathcal{H}$. 
In particular,  it holds 
$\mathcal{P}_{0,0}^{2k}(\mathcal{H})
=
\mu_{2k}\left(H\right)
=
\mu_k\left(T\right)$. 
Note that all terms equal 0 in the infinite sum over $z$ (in the first line above) if the upper score in one of the three binomial coefficients is smaller than the lower one. 
This leads to the condition that $z\leq \left\lfloor \frac{k}{p+q} \right\rfloor$.
Furthermore, it holds

\begin{align*}
&\sum_{j=1}^{\left\lfloor \frac{k}{p+q} \right \rfloor} 
\sum_{z=jq}^{k-jp} 
\binom{k}{z}
\binom{k}{jp+z}
\binom{z}{jq}
\frac{(jq)!(2z+jp)!}{z!(z+j(p+q))!} \\
=&
\sum_{j=1}^{\left\lfloor \frac{k}{p+q} \right\rfloor}
\sum_{z=jq}^{k-jp} 
\frac{(k!)^2(2z+jp)!}{z! (k-z)!(jp+z)!(k-jp-z)!(z-jq)!(z+j(p+q))!}\\
=&
\sum_{j=1}^{\left\lfloor \frac{k}{p+q} \right\rfloor}
\sum_{z=0}^{k-j(p+q)} 
\frac{(k!)^2(jp+2(z+jq))!}{(z+jq)! (k-z-jq)!(j(p+q)+z)! z!(z+j(p+2q))!}\cdot 
\frac{1}{(k-j(p+q)-z)!}\\
=&
\sum_{j=1}^{\left\lfloor \frac{k}{p+q} \right\rfloor}
\sum_{z=0}^{k-j(p+q)} 
\frac{(jp+2(z+jq))! (k-j(p+q)-z)! (k!)^2}{(z+jq)! (k-z-jq)! z!(z+j(p+2q))! (j(p+q)+z)!(2(k-j(p+q)-z))!}\\
&\hspace*{2cm} \times\binom{2(k-j(p+q)-z)}{k-j(p+q)-z} \\
=&
\sum_{j=1}^{\left\lfloor \frac{k}{p+q} \right\rfloor}
\sum_{z=0}^{k-j(p+q)} 
\frac{(jp+2(z+jq))!(k-k_1)!^3 k_1!}{(z+jq)!(k-z-jq)!z!(z+j(p+2q))!(2(k-k_1))!} 
\sum_{k_2=0}^{k-k_1} 
\binom{k}{k_1,k_2,k-k_1-k_2}^2\\
=&
\sum_{j=1}^{\left \lfloor \frac{k}{p+q} \right \rfloor} 
\sum_{k_1=j(p+q)}^{k}
\sum_{k_2=0}^{k-k_1} 
\frac{(2k_1-jq)!(k-k_1)!^3 k_1!}{(k_1-jq)!(k-k_1+jq)!(k_1-j(p+q))!(k_1+jp)!(2(k-k_1))!} 
\binom{k}{k_1,k_2,k-k_1-k_2}^2\\
=&
\sum_{k_1+k_2+k_3=k}
\sum_{j=1}^{\left\lfloor \frac{k_1}{p+q} \right\rfloor} 
\frac{k!}{k!}\cdot 
\frac{(k-k_1)!^3 k_1! (2k_1-jq)!}{(k_1-jq)! (k-k_1+jq)!(2(k-k_1))! (k_1-j(p+q))! (k_1+jp)!}
\binom{k}{k_1,k_2,k_3}^2 \\
=&
\sum_{k_1+k_2+k_3=k}
\sum_{j=1}^{\left\lfloor\frac{k_1}{p+q}\right\rfloor} 
\frac{k!(k-k_1)!^3 k_1!}{k! (k_1-jq)!(k-k_1+jq)!(2(k-k_1))!}
\binom{2k_1-jq}{k_1+jp}
\binom{k}{k_1,k_2,k_3}^2\\
=&
\sum_{k_1+k_2+k_3=k}
\binom{k}{k_1,k_2,k_3}^2
\sum_{j=1}^{\left\lfloor\frac{k_1}{p+q}\right\rfloor} 
\frac{\binom{2k_1-jq}{k_1+jp}\binom{k}{k_1-jq}}{\binom{k}{k_1}\binom{2(k-k_1)}{k-k_1}},
\end{align*}
where we used the following notation: $k_1:=j(p+q)+z$.
\end{proof}

\begin{proposition}[Proof of representation \eqref{eq:moments3}]
It holds
\begin{equation*}
\mu_k\left( \Lambda^{*}_{p,q} \right) 
= 
\sum_{k_1+\ldots+k_7=k}
\binom{k}{k_1,\ldots,k_7} 3^{k_1} 
\mathbbm{1}
\left\{
\begin{array}{l}
pk-k_2-k_4+k_5+k_7 \equiv 0 \text{ mod }p\\
p\left(-k_2 + k_3 + k_5 - k_6\right) = q\left( k_2 + k_4 - k_5 - k_7 \right)
\end{array}
\right\},
\quad k\in\N_0.
\end{equation*}
\end{proposition}

\begin{proof}
It holds
\begin{align*}
\mu_k\left( \Lambda^{*}_{p,q} \right)
&=\frac{1}{(2\pi)^3} \sum_{j \in \Z} \int_{-\pi}^\pi \int_{-\pi}^\pi \int_{-\pi}^\pi \left| e^{i\varphi_0} + e^{i\varphi_1} + e^{i\varphi_2} \right|^{2k} e^{ij\left(p(\varphi_0-\varphi_2)+q(\varphi_1-\varphi_2) \right)} \diff \varphi_0 \diff \varphi_1 \diff \varphi_2 \\
&= \frac{1}{(2\pi)^3} 
\int_{-\pi}^\pi \int_{-\pi}^\pi \int_{-\pi}^\pi 
\left| e^{i(\varphi_0 -\varphi_1)} + 1 + e^{i(\varphi_2-\varphi_1)} \right|^{2k} \sum_{j=-k}^k 
e^{ij\left(p(\varphi_0-\varphi_2)+q(\varphi_1-\varphi_2) \right)} 
\diff \varphi_0 \diff \varphi_1 \diff \varphi_2,
\end{align*}
since the triple integral in the sum vanishes for all $\left| j\right|>\lfloor \frac{k}{5} \rfloor$.
Next, apply the change of variables $\theta_1 := \varphi_0 -\varphi_1,\quad \theta_2 := \varphi_2-\varphi_1$.
Integrating over $\varphi_1$, it follows
\begin{align*}
\mu_k\left( \Lambda^{*}_{p,q} \right) &= \frac{1}{(2\pi)^2}  \int_{-\pi}^\pi \int_{-\pi}^\pi \left| e^{i\theta_1} + 1 + e^{i\theta_2} \right|^{2k} \sum_{j=-k}^k e^{ij\left(p(\theta_1-\theta_2) - q \theta_2  \right)} \diff \theta_1 \diff \theta_2.
\end{align*}
Use a second change of variables $z := e^{i\theta_1},\quad w:= e^{-i\theta_2}$,
with 
\begin{align*}
\diff z = ie^{i\theta_1} \diff \theta_1 = iz\diff \theta_1,\quad \quad
\diff w = -ie^{-i\theta_2} \diff \theta_2 = -iw\diff \theta_2, 
\diff \theta_1 \diff \theta_2 
=
\frac{1}{zw}\diff z \diff w. 
\end{align*}
Note that the first contour integral w.r.t. $z$ goes counterclockwise and the other contour clockwise.\\ 
Denote by $B_r(o)$ the complex circle with radius $r$, i.e., $B_r(0)=\lbrace re^{i\theta}­­ ~|~  \theta\in \left[0,2\pi\right] \rbrace$. 
Since $\left| z \right|^2 = z \overline{z}$ for all $z\in\C$, the substitution yields the following integral
\begin{align*}
\mu_k\left( \Lambda^{*}_{p,q} \right) 
&= 
\frac{1}{(2\pi)^2} \oint_{B_1(0)} \oint_{B_1(0)} \left( \left(z+1+\frac{1}{w}\right)\left(\frac{1}{z}+1 +w \right) \right)^k 
\frac{1}{zw}
\sum_{j=-k}^k \left(z^{p}w^{p+q}\right)^{j} \diff z\diff w\\
&= 
\frac{1}{(2\pi)^2} \oint_{B_1(0)} \oint_{B_1(0)} 
\frac{\left(w (z+1) + 1\right)^k \left(z (w + 1) +1 \right)^k}{(zw)^{k+1}} \sum_{j=-k}^k \left(z^{p}w^{p+q}\right)^{j}
\diff z\diff w \notag \\
&= 
\frac{1}{(2\pi)^2} \oint_{B_1(0)} 
\frac{1}{w^{k (p+q+1) + 1}} 
\oint_{B_1(0)} 
\underbrace{
\frac{(w(z+1)+1)^k (z(w+1)+1)^k}{z^{ k (p+1) + 1}}\cdot 
\sum_{j=0}^{2k}\left( z^p w^{p+q} \right)^j
}_{=:f(z)}  
\diff z \diff w. \notag
\end{align*}
The function $f$ has only one pole of order $k (p+1) + 1$ at $z = 0$. 
We need to compute its residue:
\begin{align*}
\text{res} f(0) 
&= 
\frac{1}{(k(p+1))!}
\lim_{z\rightarrow 0} 
\frac{\diff^{k(p+1)}}{\diff z^{k(p+1)}} 
\left\{ z^{k(p+1)+1}f(z)\right\} \\
&=  
\frac{1}{(k(p+1))!}
\lim_{z\rightarrow 0} 
\frac{\diff^{k(p+1)}}{\diff z^{k(p+1)}}  
(w (z+1) + 1)^k (z(w+1) + 1)^k \cdot 
\sum_{j=0}^{2k} 
\left(z^p w^{p+q}\right)^j \\
&= 
\frac{1}{(k(p+1))!}
\lim_{z\rightarrow 0}
\sum_{j=0}^{2k}  
\frac{\diff^{k(p+1)}}{\diff z^{k(p+1)}} 
\left((zw)^2 + z^2w + zw^2 + 3zw + w + z + 1 \right)^k 
\left( z^{p} w^{p+q} \right)^{j}\\
&= 
\frac{1}{(k(p+1))!}
\lim_{z\rightarrow 0} 
\sum_{j=0}^{2k} 
\frac{\diff^{k(p+1)}}{\diff z^{k(p+1)}} \\
& \hspace*{2.1cm}
\times\sum_{k_1+\ldots +k_7=k} \binom{k}{k_1,\ldots,k_7} 3^{k_1} w^{2 (k_3 + k_4) + k_1 + k_2 + k_5} z^{2 (k_2 + k_4) + k_1 + k_3 + k_6} \cdot
\left( z^{p} w^{p+q} \right)^{j}\\
&= 
\frac{1}{(k(p+1))!} 
\sum_{j=0}^{2k} 
\sum_{k_1+\ldots +k_7=k}   
\binom{k}{k_1,\ldots,k_7} 3^{k_1} w^{2(k_3+k_4)+k_1+k_2+k_5+(p+q)j}\\
& \hspace*{2.1cm}
\times\lim_{z\rightarrow 0} 
\frac{\diff^{k(p+1)}}{\diff z^{k(p+1)}} 
z^{2 (k_2 + k_4) + k_1 + k_3 + k_6 + pj} \\
&= 
\frac{1}{(k(p+1))!} 
\sum_{j=0}^{2k} 
\sum_{k_1+\ldots +k_7=k}   
\binom{k}{k_1,\ldots,k_7} 3^{k_1} 
w^{2(k_3+k_4)+k_1+k_2+k_5+(p+q)j} \\
&\hspace*{2.1cm} \times 
\frac{(2 (k_2 + k_4) + k_1 + k_3 + k_6 + pj)!}{(2 (k_2 + k_4) + k_1 + k_3 + k_6 + pj - k(p+1))!}
\lim_{z\rightarrow 0}  
z^{2 (k_2 + k_4) + k_1 + k_3 + k_6 + pj - k(p+1)}\\
&= 
\sum_{j=0}^{2k} 
\sum_{k_1+\ldots +k_7=k} 
\binom{k}{k_1,\ldots,k_7} 
\binom{2 (k_2 + k_4) + k_1 + k_3 + k_6 + pj}{k(p+1)} 
3^{k_1} 
w^{2(k_3+k_4)+k_1+k_2+k_5+(p+q)j}\\
&\hspace*{2.1cm}  \times 
\lim_{z\rightarrow 0} 
z^{2 (k_2 + k_4) + k_1 + k_3 + k_6 + pj - k(p+1)}.
\end{align*}
The terms in this sum vanish if
\begin{align*}
2 (k_2 + k_4) + k_1 + k_3 + k_6 + pj - k(p+1) \not= 0.
\end{align*}
Hence, we only care about the case when
\begin{align*}
&2 (k_2 + k_4) + k_1 + k_3 + k_6 + pj - k(p+1) = 0\\
\Leftrightarrow~~ 
&pj = k(p+1) - 2(k_2 + k_4) - k_1 - k_3 - k_6 = pk - k_2 - k_4 + k_5 + k_7
\end{align*}
holds. 
We can include this condition into the first sum if we change the order of summation:
\begin{align*}
\text{res} f(0) 
= 
&\sum_{k_1+\ldots +k_7=k}
\binom{k}{k_1,\ldots,k_7} 3^{k_1}\\
\times &\sum_{\substack{j=0, \\ j = \frac{pk-k_2-k_4+k_5+k_7}{p}}}^{2k} 
 \binom{2 (k_2 + k_4) + k_1 + k_3 + k_6 + pj}{k(p+1)}  w^{2(k_3+k_4)+k_1+k_2+k_5 + (p+q)j}\\
= 
\sum_{k_1+\ldots +k_7=k} 
&\binom{k}{k_1,\ldots,k_7} 
3^{k_1} 
w^{k(p+q+1) - k_2 + k_3 + k_5 - k_6 +\frac{q}{p}\left(k_5+k_7-k_2-k_4\right)} 
\mathbbm{1}\lbrace pk-k_2-k_4+k_5+k_7 \equiv 0\text{ mod } p \rbrace.
\end{align*}
Hence, we get for the inner integral of our moment formula
\begin{align*}
&\oint_{B_1(0)} f(z) \diff z 
= 
2\pi i \cdot \text{res} f(0) \\
= 
&2\pi i 
\sum_{k_1+\ldots +k_7=k} 
\binom{k}{k_1,\ldots,k_7} 3^{k_1} 
w^{k(p+q+1) - k_2 + k_3 + k_5 - k_6 +\frac{q}{p}\left(k_5+k_7-k_2-k_4\right)} 
\mathbbm{1}\lbrace pk-k_2-k_4+k_5+k_7 \equiv 0\text{ mod } p \rbrace.
\end{align*}
Then, it holds for the moments 
\begin{align*}
\mu_k\left( \Lambda^{*}_{p,q} \right) 
&= 
\frac{i}{2\pi} 
\sum_{k_1+\ldots +k_7=k} 
\binom{k}{k_1,\ldots,k_7} 
3^{k_1}   
\mathbbm{1}\lbrace pk-k_2-k_4+k_5+k_7 \equiv 0\text{ mod } p \rbrace \\
& \times 
\oint_{B_1(0)} 
w^{- k_2 + k_3 + k_5 - k_6 +\frac{q}{p}\left(k_5+k_7-k_2-k_4\right)-1} 
\diff w .
\end{align*}
Recall that it holds
\begin{align*}
\oint_{B_1(0)} 
w^{x} 
\diff w
=
\begin{cases}
-2\pi i,\quad &\text{ if }x=-1,\\
0, &\text{ else.}
\end{cases}
\end{align*}
Hence, we need to understand when the power of $w$ equals $-1$. 
It holds
\begin{align*}
&- k_2 + k_3 + k_5 - k_6 +\frac{q}{p}\left(k_5+k_7-k_2-k_4\right)-1=-1\\
\Leftrightarrow\quad 
&p\left( -k_2 + k_3 + k_5 - k_6\right) = q\left( k_2 + k_4 - k_5 - k_7\right),
\end{align*}
which yields the claim.
\end{proof}

\begin{proposition}[Proof of relation \eqref{eq:I0_integral}]\label{prop:I0_integral}
It holds
\begin{align*}
I_0\left(\sqrt{\alpha^2+\beta^2}\right) 
= 
\frac{1}{2\pi} \int_0^{2\pi} \exp\left(\alpha\cos \varphi + \beta \sin \varphi\right)
\diff \varphi,\quad \alpha,\beta\in\R.
\end{align*}
\end{proposition}
\begin{proof}
Starting from the right-hand side of the above equation, we see
\begin{align*}
\frac{1}{2\pi} \int_0^{2\pi} \exp\left(\alpha\cos \varphi + \beta \sin \varphi\right)
\diff \varphi
&=
\frac{1}{2\pi} \int_0^{2\pi} \exp\left(\alpha\cos\left( \varphi + \arctan\left( \frac{\beta}{\alpha}\right) \right)+ \beta \sin \left(\varphi + \arctan\left(\frac{\beta}{\alpha}\right)\right)\right)
\diff \varphi\\
&=
\frac{1}{2\pi}\int_0^{2\pi} \exp\left(\cos(\varphi)\sqrt{\alpha^2+\beta^2}\right) \diff \varphi\\
&=
2\cdot \frac{1}{2\pi}\int_0^{\pi} \exp\left(\cos(\varphi)\sqrt{\alpha^2+\beta^2}\right)\diff \varphi\\
&=
\frac{1}{\pi}\int_0^{\pi} \exp\left(\cos(\varphi)\sqrt{\alpha^2+\beta^2}\right)\diff \varphi
=
I_0\left(\sqrt{\alpha^2 + \beta^2}\right),
\end{align*}
since 
\begin{align*}
&\alpha \cos\left(\varphi + \arctan\left(\frac{\beta}{\alpha}\right)\right)
=
\alpha\cos(\varphi)\underbrace{\cos\left(\arctan\left(\frac{\beta}{\alpha}\right)\right)}_{=\frac{1}{\sqrt{\frac{\beta^2}{\alpha^2}+1}}}
-\alpha\sin\left(\varphi\right)\underbrace{\sin\left(\arctan\left(\frac{\beta}{\alpha}\right)\right)}_{=\frac{\beta}{\alpha\sqrt{\frac{\beta^2}{\alpha^2}+1}}},\\
&\beta\sin\left(\varphi + \arctan\left(\frac{\beta}{\alpha}\right)\right)
=
\beta\sin\left(\varphi\right)\underbrace{\cos\left(\arctan\left(\frac{\beta}{\alpha}\right)\right)}_{=\frac{1}{\sqrt{\frac{\beta^2}{\alpha^2}+1}}}
+ \beta\cos\left(\varphi\right)\underbrace{\sin\left(\arctan\left(\frac{\beta}{\alpha}\right)\right)}_{=\frac{\beta}{\alpha\sqrt{\frac{\beta^2}{\alpha^2}+1}}}\\
\Rightarrow \quad 
&\alpha \cos\left(\varphi + \arctan\left(\frac{\beta}{\alpha}\right)\right) + \beta\sin\left(\varphi + \arctan\left(\frac{\beta}{\alpha}\right)\right)\\
=
&\cos(\varphi) \frac{\alpha+\beta^2}{\sqrt{\frac{\beta^2}{\alpha^2}+1}}
+\sin(\varphi)\left(\frac{\beta}{\sqrt{\frac{\beta^2}{\alpha^2}+1}}-\frac{\beta}{\sqrt{\frac{\beta^2}{\alpha^2}+1}}\right)\\
= &\cos(\varphi)\sqrt{\alpha^2+\beta^2}.
\end{align*}
\end{proof}

\begin{construction}\label{construction:55--nanotube}
A finite dual $(5,5)$--nanotube can be constructed by combining two copies of the cap shown on the left of Figure \ref{fig:N55_construction}, with $r\in\N_0$ hexagonal belts, as depicted on the right in Figure \ref{fig:N55_construction}. 
Each cap contains six vertices of degree $5$ and ten vertices of degree $6$. 
The hexagonal ring consists of $10$ vertices.
Therefore, for a finite dual  $(5,5)$--nanotube with $r$ hexagonal rings, the total number of vertices in the original fullerene is given by
\begin{align*}
n = 60 + 20r.
\end{align*}

\begin{figure}[H]
\begin{center}
\includegraphics[scale=1.2]{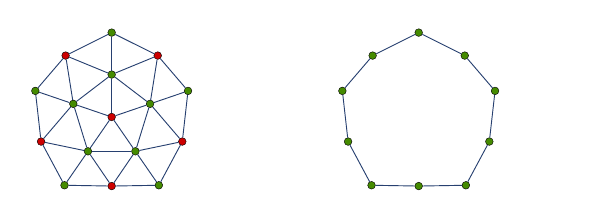}
\caption{Left: Cap of a finite dual $(5,5)$--nanotube. Right: (Dual) Hexagonal ring, which can be appended arbitrarily often to the cap.}\label{fig:N55_construction}
\end{center}
\end{figure}
For $r=0$, i.e. without any hexagonal ring and representing the smallest possible finite $(5,5)$--nanotube, the two caps can be directly connected, as shown in Figure \ref{figure:C60_caps}, resulting in the \textit{Buckminster fullerene} $C_{60,1812}$.
\begin{figure}[H]
\begin{center}
\includegraphics[scale=0.35]{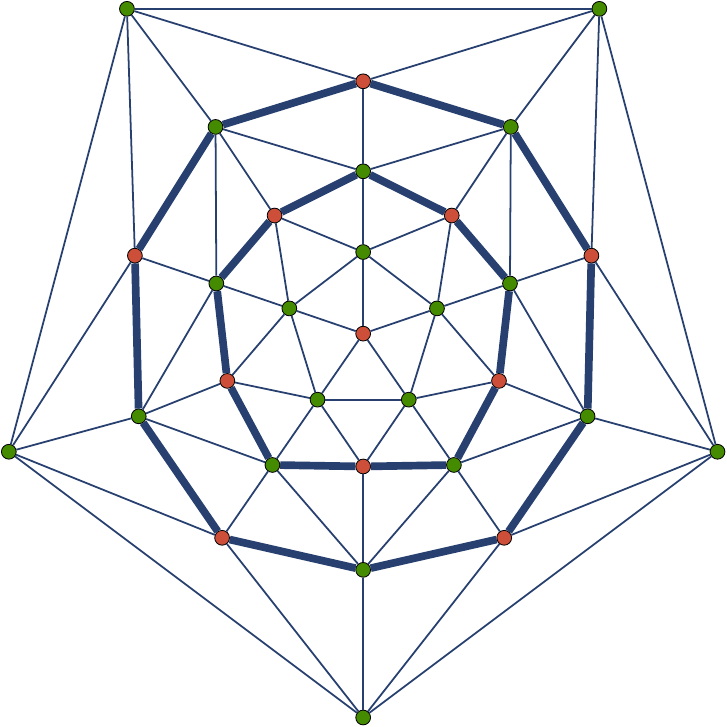}
\caption{Dual graph of the Buckminster fullerene, highlighting two copies of the cap shown in Figure \ref{fig:N55_construction}. The outer edges of the two caps are shown with thicker lines.}\label{figure:C60_caps}
\end{center}
\end{figure}
\end{construction}

%\newpage
%\input{notes}
%\newpage

%\input{notes}

%\input{appendix}

\end{document}